\documentclass[11pt]{amsart}

\usepackage{amsmath}
\usepackage{amsfonts}
\usepackage{amsthm}
\usepackage{amssymb}
\usepackage{mathrsfs}
\usepackage{verbatim}
\usepackage{enumerate}
\usepackage{graphicx}
\usepackage{longtable}
\usepackage[all,cmtip]{xy}
\usepackage{upref}
\usepackage{hyperref}
\usepackage[hang,small,bf]{caption}
\usepackage{marginnote}
\usepackage{a4wide}
\usepackage{mathtools}
\usepackage{xfrac}
\usepackage{color}
\usepackage{bm}
\newtheorem{thm}{Theorem}[section]
\newtheorem{prp}[thm]{Proposition}
\newtheorem{lem}[thm]{Lemma}
\newtheorem{cor}[thm]{Corollary}

\theoremstyle{definition}
\newtheorem{dfn}[thm]{Definition}

\newtheorem*{rmk*}{Remark}

\numberwithin{equation}{section}

\newcommand{\Z}{{\mathbb Z}}
\newcommand{\Q}{{\mathbb Q}}
\newcommand{\R}{{\mathbb R}}
\newcommand{\N}{{\mathbb N}}

\newcommand{\wt}{\widetilde}

\begin{document}

\title[Transfer products on loop spaces]{Homology transfer products on free loop spaces: orientation reversal on spheres}

\author{Philippe Kupper}

\address{Fakult\"at f\"ur Mathematik, KIT, Englerstraße 2, 
76131 Karlsruhe, Germany}

\email{philippe.kupper@kit.edu}

\begin{abstract}
We consider the space $\Lambda M\coloneqq H^1(S^1,M)$ of loops of Sobolev class $H^1$ of a compact smooth manifold $M$, the so-called free loop space of $M$. We take quotients $\Lambda M/G$ where $G$ is a finite subgroup of $O(2)$ acting by linear reparametrization of $S^1$. We use the existence of transfer maps $tr\colon H_*(\Lambda M/G)\to H_*(\Lambda M)$ to define a homology product on $\Lambda M/G$ via the Chas-Sullivan loop product. We call this product $P_G$ the transfer product. The involution $\vartheta\colon\Lambda M\to \Lambda M$ which reverses orientation, $ \vartheta\big(\gamma(t)\big)\coloneqq\gamma(1-t)$, is of particular interest to us. We compute $H_*(\Lambda S^n/\vartheta;\mathbb{Q})$, $n>2$, and the product $P_\vartheta\colon H_i(\Lambda S^n/\vartheta;\mathbb{Q})\times H_j(\Lambda S^n/\vartheta;\mathbb{Q})\to H_{i+j-n}(\Lambda S^n/\vartheta;\mathbb{Q})$ associated to orientation reversal. Rationally $P_\vartheta$ can be realized "geometrically" using the concatenation of equivalence classes of loops. There is a qualitative difference between the homology of $\Lambda S^n/\vartheta$ and the homology of $\Lambda S^n/G$ when $G\subset S^1\subset O(2)$ does not "contain" the orientation reversal. This might be interesting with respect to possible differences in the number of closed geodesic between non-reversible and reversible Finsler metrics on $S^n$, the latter might always be infinite.
\end{abstract}

\maketitle

\tableofcontents

\section*{Introduction}
In this article we investigate the topology of the free loop space $\Lambda M$ of a compact smooth manifold $M$. More precisely, we are interested in the homology of the quotient space $\Lambda M/G$ where $G$ is a finite discrete group acting on $\Lambda M$.

Our motivation is geometric in nature, namely, given a Riemannian metric $g$ on $M$, its closed geodesics are elements of the free loop space $\Lambda M$. Moreover, $\Lambda M$ possesses the structure of a smooth, infinite dimensional manifold and there is a smooth function $E_g\colon\Lambda M\to\R$, called the energy function associated to a metric $g$ on $M$, whose nonconstant critical points are exactly the closed geodesics of $g$. Furthermore, there is a Riemannian metric $g_1$ on $\Lambda M$ such that the function $E_g$ satisfies the so-called Palais-Smale condition. This enables us to relate the topology of $\Lambda M$ to the critical points of $E_g$ via Lusternik-Schnirelmann or even Morse theory. If we want to find and count the closed geodesic of a metric $g$ it might thus make sense to study $H_*(\Lambda M)$. 

There is a natural isometric action of $O(2)$ on $(\Lambda M,g_1)$ which leaves the energy $E_g$ invariant for any metric $g$. This action is induced by linear reparametrization of $S^1$. Hence, we do not lose information about closed geodesics if we look at the quotient spaces $\Lambda M/G$ for subgroups $G$ of $O(2)$. A action of special interest is that of $\Z_2\subset O(2)$ induced by the involution
\begin{equation*}
    \vartheta\colon\Lambda M\to \Lambda M,\, \vartheta\left(\gamma\right)(t)\coloneqq\gamma(1-t),
\end{equation*}
where we view a loop $\gamma$ to be parametrized on the interval $[0,1]$. This is the involution that reverses the orientation of a loop. We sometimes also use the notation $\bar{\gamma}$ to denote $\gamma$ with opposite orientation: $\bar{\gamma}\coloneqq\vartheta(\gamma)$.

We are going to look at quotients $\Lambda M/G$ for finite subgroups $G$ of $O(2)$. We will see that, for some coefficients $R$ (not necessarily the optimal ones) such as $\Q$, the homology of the quotient is the same as the homology of $\Lambda M$ if we "divide" only by subgroups of $S^1\subset O(2)$. In contrast, as soon as we look at subgroups of $O(2)$ "involving" the orientation reversal $\vartheta$ above, the homology changes and $H_*(\Lambda M;R)$ and $H_*(\Lambda M/G;R)$ are not isomorphic any longer. This is interesting for the following geometric reason:
The critical-point-theoretic approach to the closed geodesic problem is not only possible in the Riemannian setting but also in the more general framework of Finsler Geometry. However, the closed geodesics of a Finsler metric might not be invariant under orientation reversal, the energy of $c$ and of $\vartheta(c)$ might be different. This is best exemplified when $M=S^n$: On $(S^n,g_{st})$, where $g_{st}$ is the standard, round metric, there are infinitely many closed geodesics, namely the great circles parametrized proportionally to arc length. Certainly, this is due to the huge isometry group of $g_{st}$, but it is also known that "most" Riemannian metrics (on compact, simply-connected manifolds) have infinitely many closed geodesics. Results in this direction are Theorems 9.8, 9.10 and 9.11 in \cite{MR1237191}. Of course, here we mean infinitely many "geometrically distinct" ones, so if we count a closed geodesic $c$ we do not count its iterates $c^r(t)\coloneqq c(rt),\,r\in\N$, which are different curves "representing" the same closed geodesic. On spheres $S^n$ we have no examples of Riemannian metrics with only finitely many closed geodesics, but there are examples of Finsler metrics which have only finitely many. Those are the so-called Katok examples (\cite{MR743032}). These Katok metrics are non-reversible metrics, so their energy is not invariant under $\vartheta$. Bearing this in mind, this change in topology when "dividing" $\Lambda M$ by $\vartheta$ might be of geometric relevance.\\

\paragraph*{Organizaion of the paper:}
\begin{enumerate}
    \item In the first section we recall the definition of the Chas-Sullivan loop product, which is defined on the homology $H_*(\Lambda M)$ of the free loop space $\Lambda M$ of a compact smooth manifold $M$. The product uses the concatenation of loops, i.e. a map $\phi\colon\mathcal{F}\coloneqq\{(\alpha,\beta)\in\Lambda\times\Lambda|\,\alpha(0)=\beta(0)\}\to \Lambda M$. Here $\mathcal{F}$ denotes the so-called figure-eight space, a submanifold of $\Lambda M\times \Lambda M$ of codimension $dim(M)$. We also recall the explicitly known Chas-Sullivan algebra of spheres $S^n$, $n>2$.
        
    \item In the second section we define the transfer maps $tr\colon H_i(\Lambda M/G)\to H_i(\Lambda M)$ for quotients of actions of finite discrete groups $G$. Using this we define what we call the transfer product $P_G$, which is a product on the homology $H_i(\Lambda M/G)$. It is defined to be the composition
\begin{equation*}
\xymatrixcolsep{5pc}
 \xymatrix{
  H_i(\Lambda/G;R)\times H_j(\Lambda/G;R)\ar[d]_-{\text{transfer}\times\text{transfer}}&H_{i+j-n}(\Lambda/G;R)\\
  H_i(\Lambda;R)\times H_j(\Lambda;R)\ar[r]_-{\text{Chas-Sullivan product}}&
  H_{i+j-n}(\Lambda;R)\ar[u]_-{q_*}
    }
\end{equation*}
    where $q\colon\Lambda M\to \Lambda M/G$ is the quotient map.
    
    \item In Section 3 we compute the transfer product $P_\vartheta$ associated to the orientation reversal of loops, $\left(\gamma\right)(t)\coloneqq\gamma(1-t)$, on spheres $S^n$, $n>2$ with rational coefficients. We in fact deduce the structure of the algebras $(H_*(\Lambda S^n;\Q),P_G)$ for all groups $G$ conjugate to some dihedral group.\\
\end{enumerate}

We also consider a $\Z_2$-action which is closely related to the action $\vartheta$, namely the $\Z_2$-action induced by the involution $\theta\colon\Lambda M\to\Lambda M$
\begin{equation*}
    \theta(\gamma)(t)\coloneqq\gamma\left(1-(t+\tfrac{1}{2})\right).
\end{equation*}
We denote the action also simply by $\theta$.\\

\paragraph*{Results:}

\begin{itemize}
    \item For $G\subset S^1\subset O(2)$ finite we always have
    \begin{equation*}
        H_i(\Lambda M/G;\Q)\cong H_i(\Lambda M;\Q)
    \end{equation*}
    as $\Q$-modules for all $i$. While for finite $G\subset O(2)$ with $G\cap\big(O(2)\setminus S^1)\neq \emptyset$ we have
    \begin{equation*}
        H_i(\Lambda M/G;\Q)\cong H_i(\Lambda M/\vartheta;\Q)\cong H_i(\Lambda M;\Q)^\vartheta\coloneqq\{x\in H_i(\Lambda M;\Q)|\;\vartheta_*(x)=x\}
    \end{equation*}
    (compare Proposition \ref{reversible vs. nonreversible?}).
    
    \item Under orientation reversal the fibration $ev_0:\Lambda M\rightarrow M,\,\gamma\mapsto \gamma(0)$ descends to a fibration $ev_0/\vartheta:\Lambda M/\vartheta\rightarrow M$ with fibre $\Omega M/\vartheta$ where $\Omega M$ denotes the based loop space. Rationally $ev_0/\vartheta$ induces an algebra homomorphism up to scaling between $P_\vartheta$ and the intersection algebra on $M$.
\begin{equation*}
\xymatrixrowsep{1.5pc}
\xymatrixcolsep{4pc}
\xymatrix{
\Omega\ar[r]^-j\ar[d]^-q&\Lambda\ar[r]^-{ev_0}\ar[d]^-q&M\ar[d]\\
\Omega/\vartheta\ar[r]\ar[r]^-{j/\vartheta}&\Lambda/\vartheta\ar[r]^-{ev_0/\vartheta}& M,
 }
\end{equation*}
    Also the quotient of the fibre inclusion $j/\vartheta$ induces an algebra homomorphism $(j/\vartheta)_!$ between $P_\vartheta$ and the transfer product on $\Omega/\vartheta$ (induced by the Pontrjagin product) when rational coefficients are used (see Propositions \ref{ev/G homo}, \ref{j/G homo}).
    
    \item For a finite subgroup $G$ of $O(2)$, acting by linear reparametrization on $\Lambda M$, the quotient map $q\colon\Lambda M\to \Lambda M/G$ induces an algebra isomorphism up to scaling between the rational Chas-Sullivan algebra restricted to classes that are invariant under the action of $G$ and the transfer algebra $\big(H_*(\Lambda M;\Q),P_G\big)$ (see Theorem \ref{O(2) subgroup iso}).\\
Furthermore, we compute the product
\begin{equation*}
    P_\vartheta\colon H_i(\Lambda S^n/\vartheta;\Q)\times H_j(\Lambda S^n/\vartheta;\Q)\to H_{i+j-n}(\Lambda S^n/\vartheta;\Q)
\end{equation*}
and in particular show that there are nonnilpotent classes for this product on spheres (see Theorem \ref{main thm}). The existence of nonnilpotent classes could be useful in the search for closed geodesics (compare \cite{hingston2013}). An analogous result holds for the product $P_\theta$ with rational coefficients.
\end{itemize}

\begin{rmk*}
The products $P_\vartheta$ and $P_\theta$ possess "geometric realizations". By that we mean that there are homology products $A_\vartheta$ on $H_*(\Lambda M/\vartheta)$ and $A_\theta$ on $H_*(\Lambda M/\theta)$ which are defined using the concatenation of equivalence classes of pairs of loop, i.e. using maps $\mathcal{F}/\Z_2\to \Lambda M/\vartheta$ and $\mathcal{F}/\Z_2\to \Lambda M/\theta$ respectively, such that 
\begin{align*}
A_\vartheta=P_\vartheta,\quad A_\theta=P_\theta
\end{align*}
holds up to sign. For the definition of the products $A_\vartheta$ and $A_\theta$ the author constructs "tubes" around $\mathcal{F}/\Z_2$ as a subset of $\Lambda\times\Lambda/\Z_2$. While in the case of the action $\theta$ these neighbourhoods are actual tubular neighbourhood, in the case of $\vartheta$ the neighbourhoods are only homeomorphic to a quotient of a vector bundle. These "geometric products" are constructed in \cite{kupper20}.
\end{rmk*}

This article is based on the author's Ph.D. thesis \cite{kupper20} written at the University of Leipzig under the supervision of Hans-Bert Rademacher. The author thanks Professor Rademacher for the introduction to the closed geodesic problem and the opportunity to work on this topic. The author also thanks the referee for the very careful review and the many helpful comments improving this article.

\section{The manifold $\Lambda M$ and the Chas-Sullivan product}

\subsection{The free loop space of a manifold}
Let $M$ be a topological space. We first consider the space of continuous loops $LM\coloneqq C^0(S^1,M)$ in $M$. $LM$ carries the compact-open topology. The base point in $S^1$ in denoted by $0$. 

\begin{prp}
The evaluation map $ev_0\colon LM\to M,\, \gamma\mapsto\gamma(0)$ is a Hurewicz fibration. The fibre $ev_0^{-1}(p)$ over a point $p\in M$ is the space of loops based at $p$:
\begin{equation*}
 \Omega_p\coloneqq\{\gamma\in LM|\;\gamma(0)=p\}=ev_0^{-1}(p).
\end{equation*}  
\end{prp}

\begin{proof}
Since the inclusion of the base point $i:0\hookrightarrow S^1$ is a cofibration and $S^1$ is (locally) compact and Hausdorff we can apply \cite[Chapter VII, Theorem 6.13]{MR1224675} to deduce that $i^*=ev_0\colon LM\to C^0(0,M)\cong M$ is a Hurewicz fibration. 
\end{proof}

If $M$ is a topological manifold, then $ev_0:LM\rightarrow M$ is in addition locally trivial.\\

From now on we even take $M$ to be a smooth compact manifold. In this case, instead of investigating $LM$ we can look at the Hilbert manifold $\Lambda M$ of $H^1$-curves on $M$. Topologically this does not make much difference since $LM$ and $\Lambda M$ are homotopy equivalent. The continuous inclusions 
\begin{equation*}
 \xymatrix{
C^\infty(S^1,M)\ar@{^{(}->}[r]&\Lambda M\ar@{^{(}->}[r]&LM
     }
\end{equation*}
are in fact homotopy equivalences (\cite[Theorem 1.2.10]{MR0478069} or \cite[Theorem 1.5.1]{MR3729450}).\\

As done in \cite[page 8]{MR0478069} the set of functions $\Lambda M$ can also be described as the set of absolutely continuous maps $f\colon S^1(=[0,1]/\{0,1\})\to M$ whose derivative $f'$ (which is defined almost everywhere) is square-integrable, i.e. $\int_{S^1}g\big(f(t)\big)\big(f'(t),f'(f)\big)\mathrm{d}t<\infty$. It can be equipped with a differenttiable structure:

\begin{thm}\cite[Theorem 1.2.9]{MR0478069}
 $\Lambda M$ is a smooth Hilbert manifold of infinite dimension.
\end{thm}

$H^1$-vector fields along $H^1$-curves are the fibres of the tangent bundle $T\Lambda M\rightarrow \Lambda M$  of $\Lambda M$: $H^1\big(f^*(TM)\big)\cong T_f\Lambda M$ is the fibre over $f\in\Lambda M$ (\cite[Section 1.3]{MR0478069}). Induced by the metric $g$ on $M$, there is a metric $g_1$ on $\Lambda M$: The metric
\begin{equation*}
 g_1(X,Y)\coloneqq\int_{S^1}g(X(t),Y(t))\mathrm{d}t+\int_{S^1}g\big(\nabla_{\dot{c}}X(t),\nabla_{\dot{c}}Y(t)\big)\mathrm{d}t\,,
\end{equation*}
defined on $H^1\big(c^*(TM)\big)$ for smooth loops $c$, extends to all of $H^1(S^1,M)$ (\cite[Theorem 1.3.6]{MR0478069}).

\begin{thm} \cite[Theorem 1.4.5]{MR0478069} or \cite[Theorem 2.4.7]{MR666697} \label{complete lambda}
If $M$ is compact, $(\Lambda M, g_1)$ is a complete Riemannian manifold. That is, $\Lambda M$ with the distance induced form $g_1$ is a complete metric space.
\end{thm}

\begin{prp}\cite[Lemma 2.2]{10.1007/1-4020-4266-3_02}
The evaluation map $ev_0\colon\Lambda M\to M, \,\gamma\mapsto\gamma(0)$ is a locally trivial fibration. If $M$ is connected, the fibres are all homeomorphic to $\Omega_p$, the based loop space at any point $p\in M$. $\Omega_p\coloneqq\{\gamma\in\Lambda M|\,\gamma(0)=p\}\subset\Lambda M$ is a smooth submanifold of $\Lambda M$ of codimension $n$.  
\end{prp}

\subsection{Definition and properties of the Chas-Sullivan product}
Let $M$ now be a fixed finite-dimensional compact and connected smooth manifold. Let $\Lambda\coloneqq\Lambda M$ denote the Hilbert manifold of closed curves on $M$. Let $M$ have dimension $n$.\\
We call $\mathcal{F}\coloneqq \{(\alpha,\beta)\in\Lambda\times\Lambda|\,\alpha(0)=\beta(0)\}$ the figure-eight space.  $\mathcal{F}$ fits into the following pullback diagram:
\begin{equation*}
\xymatrixrowsep{1.5pc}
 \xymatrix{
 \mathcal{F}\ar@{^{(}->}[r]^-{i_\mathcal{F}} \ar[d]_-{ev_0} &\Lambda\times\Lambda \ar[d]^-{ev_0\times ev_0}\\
 M\ar@{^{(}->}[r]_-{\Delta}&M\times M
     }
\end{equation*}
where $ev_0\colon\mathcal{F}\to M$ is $ev_0(\alpha,\beta)\coloneqq\alpha(0)=\beta(0)$, $i_{\mathcal{F}}\colon\mathcal{F}\hookrightarrow\Lambda\times\Lambda=\Lambda^2$ is the inclusion and $\Delta\colon M\hookrightarrow M\times M=M^2,\;p\mapsto (p,p)$ is the diagonal embedding. The Chas-Sullivan product
\begin{displaymath}
 \ast\colon H_i(\Lambda)\times H_j(\Lambda)\to H_{i+j-n}(\Lambda)
\end{displaymath}
is defined as the composition (see \cite{CJ2002} or \cite[Section 5]{goresky2009})
\begin{displaymath}
\xymatrixcolsep{3pc}
  \xymatrix{
  H_i(\Lambda)\times H_j(\Lambda)\ar[d]_-{(-1)^{n(n-j)}\times}&&&H_{i+j-n}(\Lambda)\\
  H_{i+j}(\Lambda\times\Lambda)\ar[r]\ar@/^1.4pc/[rrr]^-{{i_\mathcal{F}}_!}&
  H_{i+j}(\Lambda\times\Lambda,\Lambda \times\Lambda -\mathcal{F})\ar[rr]_-{\text{Thom isomorphism}}&&
  H_{i+j-n}(\mathcal{F})\ar[u]_-{\phi_*}
  }
\end{displaymath}

Here $H_i(\cdot)$ denotes singular homology. Which coefficients we use will be clear from the context. The Chas-Sullivan product was originally defined in \cite{chas1999string} with a different definition.\\
We explain the different maps in this definition:
\begin{itemize}
 \item The first map is just the homological cross product $\times$ together with the sign-correction $(-1)^{n(n-j)}$ for nicer algebraic properties as we will see later.
 
 \item The map ${i_\mathcal{F}}_!$ is a so-called "Gysin'' or "umkehr'' map. It is defined to be the composition of the map $H_{i+j}(\Lambda\times\Lambda)\to H_{i+j}(\Lambda\times\Lambda,\Lambda \times\Lambda -\mathcal{F})$ induced by inclusion and the Thom isomorphism. That the latter exists follows from the fact that $\mathcal{F}$ is a Hilbert submanifold of $\Lambda\times\Lambda$: One can show that $ev_0\times ev_0\colon\Lambda\times\Lambda\rightarrow M\times M$ is a submersion. Thus $\mathcal{F}=(ev_0\times ev_0)^{-1}\big(\Delta(M)\big)$ is a submanifold of codimension $n$.

Let $U_{\mathcal{F}}$ be a tubular neighbourhood of $\mathcal{F}$. Since $\mathcal{F}$ is closed in $\Lambda^2=(\Lambda^2-\mathcal{F})\cup U_{\mathcal{F}}$ we have, by excision, that 
\begin{equation*}
 H_{i+j}(\Lambda^2,\Lambda^2-\mathcal{F})\cong H_{i+j}(U_{\mathcal{F}},U_{\mathcal{F}}-\mathcal{F})\cong H_{i+j}(N_{\mathcal{F}},N_{\mathcal{F}}-\mathcal{F}).
\end{equation*}
We can then cap with a (chosen) Thom class $\tau_{\mathcal{F}}\in H^n(N_{\mathcal{F}},N_{\mathcal{F}}-\mathcal{F})$ of the normal bundle $N_{\mathcal{F}}\rightarrow \mathcal{F} $, which is an isomorphism by the Thom isomorphism theorem (\cite[Appendix B]{goresky2009}) :
\begin{equation*}
 \xymatrix{
   H_{i+j}(N_{\mathcal{F}},N_{\mathcal{F}}-\mathcal{F})\ar[r]^-{\cap\tau_{\mathcal{F}}}_-\cong&H_{i+j-n}(N_{\mathcal{F}})\cong H_{i+j-n}(\mathcal{F}).
     }
\end{equation*}
Clearly, we have to take appropriate coefficients for homology here: If $M$ is orientable, then so is $N_\mathcal{F}$ (see below) and we can take any coefficient ring. Otherwise we have to take $\Z_2$ as coefficient domain.

 \item On $\mathcal{F}$ the composition of loops can be defined (\cite[Section 2]{goresky2009}): The "concatenation of loops" $\phi=\phi_{\frac{1}{2}}\colon\mathcal{F}\to\Lambda$ is a continuous map defined by
\begin{equation*}
 \phi(\gamma,\delta)(t)=\phi_{\frac{1}{2}}(\gamma,\delta)(t)\coloneqq\begin{cases} \gamma(2t), & 0\leq t\leq\frac{1}{2}  \\ \delta(2t-1), &\frac{1}{2}\leq t\leq 1 \end{cases}.
\end{equation*}
We sometimes also write $\gamma\cdot\delta$ for $\phi(\gamma,\delta)$.
 \end{itemize}

We now list some properties of the Chas-Sullivan product: The Chas-Sullivan product is 
\begin{itemize}
\item \emph{associative} (\cite[Theorem 3.3]{chas1999string} or \cite[Theorem 2.5]{1709.06839v2}).

\item \emph{graded commutative} (\cite[Proposition 5.2]{goresky2009}): we have
\begin{align*}
b\ast a=(-1)^{(|a|-n)(|b|-n)}a\ast b,
\end{align*}
where $|a|$ is the degree of $a$, i.e. $a\in H_{|a|}(\Lambda M)$.

\item \emph{unital} if $M$ is a compact oriented manifold (\cite[Theorem 2.5]{1709.06839v2}): The inclusion of point curves $c\colon M\hookrightarrow\Lambda M^0\subset\Lambda M,\;p\mapsto (t\mapsto p \;\forall t)$ is injective in homology since it is a section of the evaluation $ev_0\circ c=id_M$. The image $E:=c_*([M])$ of the orientation class is a two-sided unit element for the Chas-Sullivan product.
\end{itemize}

Let us look again at the loop fibration $\Omega\xhookrightarrow{j}\Lambda\xrightarrow{ev_0} M$. Later we are going to make use of 
\begin{prp}\label{alg hom CS Pont}\cite[Proposition 3.4]{chas1999string}
Let $M$ be a connected, compact, oriented $n$-dimensional manifold.
 \begin{enumerate}
  \item The homomorphism ${ev_0}_*\colon H_*(\Lambda;\Z)\to H_*(M;\Z)$ of graded modules is an algebra homomorphism from the Chas-Sullivan algebra to the intersection algebra.
  
  \item The homomorphism $j_!\colon H_*(\Lambda;\Z)\to H_{*-n}(\Omega;\Z)$ of graded modules is an algebra homomorphism from the Chas-Sullivan algebra to the Pontrjagin algebra.
 \end{enumerate}
\end{prp}
(compare \cite[Proposition 2.4]{10.1007/1-4020-4266-3_02}) Here $j!$ is again the Gysin map induced by the inclusion $\Omega\hookrightarrow\Lambda M$.

\subsection{Explicit tubular neighbourhoods of the figure-eight space}\label{explicit tubes}

We will need some explicit tubular neighbourhoods. We start with showing that $N_{\mathcal{F}}$, the normal bundle of $\mathcal{F}$ in $\Lambda^2$, is the pullback of the normal bundle $N_M$ of $\Delta(M)\cong M$ in $M^2=M\times M$:
\begin{lem}\label{nb is pullback}
$N_{\mathcal{F}}\cong{ev_0}^*(N_M)$.
\end{lem}

\begin{proof}
This holds since $ev_0\times ev_0\colon\Lambda^2\to M^2$ is a submersion (compare e.g. \cite[Proposition 2.4.1]{MR666697}). Therefore $(ev_0\times ev_0)^{-1}\big(\Delta(M)\big)=\mathcal{F}$ is a submanifold of codimension $n$. Also due to transversality, the composition
\begin{equation*}
\xymatrixcolsep{4pc}
\xymatrix{
 T\Lambda^2|_\mathcal{F}\ar[r]^-{d(ev_0\times ev_0)|_\mathcal{F}}&TM^2|_{\Delta(M)}\ar[r]&TM^2|_{\Delta(M)}/T\Delta(M)
 }
\end{equation*}
is fibrewise surjective. Since $T\mathcal{F}$ is in the kernel of this bundle map we get a map 
\begin{equation*}
 f\colon N_\mathcal{F}\coloneqq\frac{T\Lambda^2|_\mathcal{F}}{T\mathcal{F}}\rightarrow\frac{TM^2|_M}{TM}=:N_M
\end{equation*}
which is a surjective bundle map. For dimensional reasons $f$ is fibrewise an isomorphism. We therefore have a commutative diagram
\begin{equation*}
\xymatrixrowsep{1.5pc}
\xymatrixcolsep{4pc}
\xymatrix{
 N_\mathcal{F}\ar[rd]_-g^\cong\ar@/^2pc/[rrd]^-f\ar@/_2pc/[rdd]&&\\
 &{ev_0}^*(N_M)\ar[d]\ar[r]^-{pr_{N_M}}&N_M\ar[d]\\
 &\mathcal{F}\ar[r]^-{ev_0}&M
 }
\end{equation*}
which shows that $g$ is fibrewise an isomorphism and thus it is a bundle isomorphism over $\mathcal{F}$. That is, the normal bundle of $\mathcal{F}$ is (isomorphic to) the pullback of the normal bundle of $\Delta(M)\cong M$.
\end{proof}

Let $(M,g)$ be a compact Riemannian manifold an let $d_g(p,q)$ denote the distance in $M$ between the points $p,q\in M$ defined by the metric $g$. We define the open neighbourhood $U_{M,\varepsilon}$ of $\Delta(M)$ inside $M^2$ by $U_{M,\varepsilon}\coloneqq\{(p,q)\in M^2|\;d_g(p,q)<\varepsilon\}$. Let $\exp$ denote the Riemannian exponential map of $g$. Since $M$ is compact, there exists an $\varepsilon>0$ such that the maps 
\begin{align*}
t_M'\colon D_\varepsilon TM\hookrightarrow M\times M,\;(p,v)\mapsto \big(p, \exp(p,v)\big)
\end{align*}
embeds the disk bundle $D_\varepsilon TM\coloneqq\{(p,v)\in TM|\;|v|_g<\varepsilon\}$ as tubular neighbourhood of $\Delta{M}$ into $M^2$. The tangent bundle $TM$ and the normal bundle $N_M$ of $\Delta{M}\subset M\times M$ are isomorphic (\cite[Lemma 11.5]{MR0440554}) and we understand $t_M'$ as defined on $D_\varepsilon N_M$. We denote with $t_M$ an extension of $t_M'$ to all of $N_M$. We have $U_{M,\varepsilon}=t_M(N_M)=t_M'(D_\varepsilon N_M)$. Let $\rho$, with $0<\rho<\infty$, denote the injectivity radius of $(M,g)$. The map smooth map $h\colon U_{M,\rho}\times M \subset M\times M \times M\to M $, that ''pushes points", is defined in \cite[Lemma 2.1]{1709.06839v2} via
\begin{eqnarray*}
h(p,q)(x) \coloneqq&\left\{ 
\begin{array}{ll}
\exp\Big(p,\exp _{p}^{-1}(x)+\mu \big(d_g(p,x)\big)\exp _{p}^{-1}(q)\Big) & \text{ if }d_g(p,x)\leq \rho /2 \\ 
w & \text{ if }d_g(p,x)\geq \rho /2%
\end{array}%
\right.
\end{eqnarray*}
where $\mu \colon[0,\infty )\rightarrow \mathbb{R}$ is a smooth "cut-off" function that is constant 1 near 0, constant 0 on $[\frac{\rho}{3},\infty)$ and decays appropriately in between. $h$ satisfies
\begin{enumerate}
 \item $h(p,q)\coloneqq h(p,q,\cdot)\colon M\to M$ is a diffeomorphism if $d_g(p,q)<\frac{\rho}{14}$,
 \item $h(p,q)(p)=q$ if $d_g(p,q)<\frac{\rho}{14}$, i.e. it pushes the first point to the second if the two points a close enough,
 \item $h(p,p)=id_M$.
\end{enumerate}
The map $h$ can be used to define a tubular neighbourhood embedding $t_\mathcal{F}:N_{\mathcal{F}}\cong{ev_0}^*(N_M)\hookrightarrow\Lambda^2$ of $\mathcal{F}$ that lifts $t_M$: In \cite[Proposition 2.2]{1709.06839v2} they use $t_M'$ and construct $t_\mathcal{F}$ as follows.
We have $D_\varepsilon{ev_0}^*(N_M)={ev_0}^*(D_\varepsilon N_M)$ if we equip ${ev_0}^*(N_M)$ with the pullback metric. We choose $\varepsilon<\frac{\rho}{14}$. The embedding $t_\mathcal{F}'\colon {ev_0}^*(D_\varepsilon N_M)\hookrightarrow\Lambda^2$ is defined by 
\begin{equation*}
 t_\mathcal{F}'\big((\gamma,\delta),(x,v)\big)\coloneqq\big(\gamma,\lambda(\delta,v)\big)
\end{equation*} 
and uses the map $h$ from above that ''pushes points" on M. Here, the curve $\lambda$ is defined to be
\begin{equation*}
 \lambda(\delta,v)(t)\coloneqq h\big(\delta(0),\exp(\delta(0),v)\big)\big(\delta(t)\big)=h\big(t_M'(\delta(0),v)\big)\big(\delta(t)\big).
\end{equation*} 
Since $|v|<\varepsilon<\frac{\rho}{14}$, the distance $d_g\big(\delta(0),\exp(\delta(0),v)\big)=d_g\big(\delta(0),\lambda(0)\big)<\varepsilon<\frac{\rho}{14}$ and $h\big(\delta(0),\exp(\delta(0),v)\big)\colon M\to M$ is a diffeomorphism. Thus, instead of starting at $\delta(0)$, the curve $\lambda$ starts at $\exp(\delta(0),v)$. "$\delta$ is diffeomorphically pushed away from $\gamma$".

\begin{prp} (see \cite[Proposition 2.2]{1709.06839v2})
Let $(M,g)$ be a compact Riemannian manifold with injectivity radius $\rho$ and let $0<\varepsilon<\frac{\rho}{14}$. Then the map $t_\mathcal{F}'\colon{ev_0}^*(D_\varepsilon N_M)\hookrightarrow\Lambda^2$ is an open embedding with image $U_{\mathcal{F},\varepsilon}\coloneqq\{(\gamma,\delta)\in\Lambda^2|\,d_g\big(\gamma(0),\delta(0)\big)<\varepsilon\}$ which is an open neighbourhood of $\mathcal{F}$. Moreover, the diagram
\begin{equation*}
\xymatrixrowsep{1.5pc}
 \xymatrix{
  {ev_0}^*(D_\varepsilon N_M)\ar@{^{(}->}[r]^-{t_\mathcal{F}'} \ar[d]_{pr_{N_M}} &\Lambda\times\Lambda\ar[d]^{ev_0\times ev_0}\\
  D_\varepsilon N_M\ar@{^{(}->}[r]^-{t_M'}&M\times M
     }
\end{equation*}
commutes and $(ev_0\times ev_0)(U_{\mathcal{F},\varepsilon})=U_{M,\varepsilon}$. 
\end{prp}

Extending $t_\mathcal{F}'$ to all of ${ev_0}^*(N_M)$ yields a tubular neighbourhood map $t_\mathcal{F}$ of $\mathcal{F}\subset\Lambda\times\Lambda$.

\subsection{The Chas-Sullivan algebra of spheres}\label{CS spheres}

Let now $M=S^n$. The loop space homology of spheres is as follows:

\begin{thm} \cite{MR649625} or \cite[Section 13]{goresky2009}
If $n\geq3$, for the space $\Lambda S^n$ we have:
 \begin{itemize}
  \item for $n$ odd 
  \begin{equation*}
   H_i(\Lambda S^n;\Z)\cong\begin{cases} \Z & i=0,n \\
   \Z& i=\lambda_r=(2r-1)(n-1) \text{ for }r\in\N \\
    \Z& i=n-1+\lambda_r=2r(n-1) \text{ for }r\in\N \\
    \Z& i=n+\lambda_r=2r(n-1)+1 \text{ for }r\in\N \\
   \Z& i=2n-1+\lambda_r=n+2r(n-1) \text{ for }r\in\N \\
   \{0\} &\text{otherwise} \end{cases}
  \end{equation*}
\item for $n$ even
 \begin{equation*}
   H_i(\Lambda S^n;\Z)\cong\begin{cases} \Z & i=0,n \\
   \Z& i=\lambda_r=(2r-1)(n-1) \text{ for }r\in\N \\
    \Z_2& i=n-1+\lambda_r=2r(n-1) \text{ for }r\in\N \\
   \Z& i=2n-1+\lambda_r=n+2r(n-1) \text{ for }r\in\N \\
   \{0\} &\text{otherwise} \end{cases}
  \end{equation*}
 \end{itemize}
where $\lambda_r=(2r-1)(n-1)$.
\end{thm}
$\lambda_r$ is in fact the Morse index of a $r$-fold iterated prime closed geodesic of the round sphere $(S^n,g_{st})$ (\cite[Example 2.5.7 (iib)]{MR666697}).

The above graded modules equipped with the Chas-Sullivan product have the following algebra structure (\cite[Theorem 2]{Cohen2004}):
\begin{itemize}
 \item for $n\geq3$ and $n$ odd
 \begin{equation*}
  \big(H_{*+n}(\Lambda S^n;\Z),\ast\big)\cong\bigwedge(a)\otimes\Z[x]
 \end{equation*}
 with $|a|=-n$ and $|x|=n-1$ on the right-hand side
\item and for $n\geq2$ and $n$ even
 \begin{equation*}
  \big(H_{*+n}(\Lambda S^n;\Z),\ast\big)\cong\big(\bigwedge(s_1)\otimes\Z[a,t]\big)/(a^2,s_1\otimes a,2a\otimes t)
 \end{equation*}
 with $|a|=-n$, $|s_1|=-1$, $|t|=2n-2$ on the right-hand side.
\end{itemize}
The elements on the left-hand side are all shifted upward in degree by $n$, so that the generators on the right correspond to the generators of the free loop space homology indicated as follows:
\begin{align*}
 A\leftrightarrow a,\;\sigma_1\leftrightarrow s_1,\;U\leftrightarrow x,\;\Theta\leftrightarrow t\,,
\end{align*}
where
\begin{itemize}
 \item $A$ generates $\in H_0(\Lambda S^n;\Z)$ and corresponds to the generator of $H_0(S^n;\Z)$. It generates the intersection algebra of $S^n$: $A\bullet A=0$.
 \item $\sigma_1$ generates $H_{n-1}(\Lambda S^n;\Z)$.
 \item $E\in H_{n}(\Lambda S^n;\Z)$ corresponds to the orientation class of $S^n$. It is thus the unit of the intersection algebra $\big(H_{*+n}(S^n;\Z),\bullet\big)\cong\bigwedge(a)$ on $S^n$.  Here $a$ has degree $-n$ and is sent to $A$ under this isomorphism.
 \item $U\in H_{2n-1}(\Lambda S^n;\Z)$ exists only for odd $n$ and is sent, via $j_!$, to $x$ which generates the Pontrjagin algebra $\big(H_*(\Omega S^n;\Z),\star\big)\cong\Z[x]$ with $|x|=n-1$. 
 \item $\Theta$ generates $H_{3n-2}(\Lambda S^n;\Z)$.
\end{itemize}

It is important to use the sign correction as in \cite{goresky2009} for the definition of $\ast$ here. Otherwise the commutativity behaviour on the two sides to the algebra homomorphisms differ.\\
For us it is important to know that 
\begin{itemize}
 \item for $n\geq3$ and $n$ odd each homology class can be written as a (Chas-Sullivan) product of $A$ and $U$. 
\item and for $n\geq4$ and $n$ even each homology class can be written as a (Chas-Sullivan) product of $A$, $\sigma_1$ and $\Theta$.
\item the class $E$ is a neutral element for the Chas-Sullivan product for $n$ even and odd.
\end{itemize}
Moreover, for any $n$ there is a nonnilpotent homology class $\Theta\in H_{3n-2}(\Lambda S^n;\Z)$:
\begin{equation*}
 \Theta^k=\Theta^{\ast k}\neq0,\;\forall k\in \N.
\end{equation*}
In the odd case, $\Theta=U\ast U=U^{\ast2}$. Furthermore, multiplication with this class is an isomorphism 
\begin{equation*}
 \ast\Theta\colon H_k(\Lambda S^n;\Z)\xrightarrow{\cong} H_{k+3n-2-n}(\Lambda S^n;\Z)
\end{equation*} 
for all $k>0$. All of this (and more) is summarized in Lemma 5.4 of \cite{hingston2013}, proofs are given in \cite[Section 13]{goresky2009}.

\section{The transfer and the transfer product}

\subsection{The transfer} \label{transfer}
Let $G$ be a finite and discrete topological group acting continuously on $\Lambda=\Lambda M$, where $M$ is a compact, connected, smooth manifold of dimension $n$.

If the action is free, then the map $q\colon\Lambda\to\Lambda/G$ is a covering map and we not only have a homomorphism $q_*\colon H_i(\Lambda;R)\to H_i(\Lambda/G;R)$, but also one in the other direction $tr\colon H_i(\Lambda/G;R)\to H_i(\Lambda;R)$ for singular homology with any coefficient group $R$. This homomorphism is called transfer and it is induced by a chain map which can be defined in a very straightforward and natural way on the singular chain complexes: just send a singular simplex of $X/G$ to the sum of its $|G|$ distinct lifts. This is explained in section 3.G of \cite{HatcherAT}.

Unfortunately, no subgroup $G\subset O(2)$ (except the trivial) acts freely on $\Lambda$ and thus the quotient maps $\Lambda\to \Lambda/\vartheta$ and $\Lambda\to \Lambda/\theta$ are not covering maps. Nevertheless, also for nonfree actions of a finite group $G$ on $\Lambda$ we get transfer homomorphisms. They can be constructed in (at least) two ways:
\begin{enumerate}
\item Any orbit map $q\colon X\to X/G$ of a continuous action of a finite discrete group $G$ on a topological space $X$ is a ramified covering map in the sense of \cite{MR698352}. Let $SP^d(X)$ denote the $d$-fold symmetric product of $X$. A $d$-fold ramified covering map is a continuous surjective finite-to-one map $p\colon X\to Y$ for which a continuous map $t_p:Y\to SP^d(X)$ exists. $t_p$ maps a $y\in Y$ to the unordered $d$-tupel of points of $p^{-1}(y)$ counted with multiplicities. In \cite{MR698352}, a transfer homomorphism $tr\colon\wt{H}_i(X/G;R)\to \wt{H}_i(X;R)$ on reduced singular homology for $R=\Z,\Q,\Z_k$ is developed if $X$ and $X/G$ have the homotopy type of connected CW-complexes. It uses the homotopy-theoretic definition of homology.

Alternatively we can refer to \cite{MR809499}, where a definition of a transfer for ramified coverings as a singular chain map is given. This only requires that $X$ is Hausdorff and has the homotopy type of $CW$-complex.
\item If $X$ and $X/G$ both have the homotopy type of CW-complexes and $X$ is paracompact and Hausdorff, then by Theorem 1 of \cite{MR105677} or Corollary 10.4 of \cite{MR226631} \v{C}ech homology and singular homology of $X$ and $X/G$, respectively, are isomorphic and we can apply Theorem 7.2 of Section 7 of Chapter 3 of \cite{MR0413144} to get a transfer $tr\colon H_i(X/G;R)\to H_i(X;R)$ for singular homology with arbitrary coefficients. Since $\Lambda$ is a manifold, it is paracompact and Hausdorff and it has the homotopy type of a countable CW-complex by Corollary 2 of \cite{MR100267}. The same also holds for $\Lambda/G$ for any finite subgroup $G$ of $O(2)$ acting as explained in the introduction: In fact, $\Lambda/G$ is paracompact, Hausdorff since $q\colon\Lambda\to\Lambda/G$ is a perfect map and $\Lambda/G$ is locally contractible since the normal bundle to any orbit is an equivariant vector bundle (\cite[Chapter 1, Theorem 3.1]{MR0413144}, \cite[Theorem 4.4.15]{MR1039321}, \cite[Section 1.5]{MR739783}). Moreover, since $\Lambda$ has the $O(2)$-homotopy type of $O(2)$-CW-complex, $\Lambda/G$ has the homotopy type of countable CW-complex (\cite[Theorem 4.2]{MR997228}, \cite[Theorem A]{MR1049832}, \cite[Chapter II, Proposition 1.16]{MR889050}, \cite[Corollary 5.2.6]{MR1074175}). See also Lemma A.4 in Appendix A of \cite{goresky2009} for a comparison of \v{C}ech and singular homology groups of $\Lambda$.
\end{enumerate}

These transfer homomorphisms have similar properties as the transfer of actual coverings. We now list of some of their properties and we are going to use them later (\cite[Section 2]{MR698352}, \cite[Chapters 3 and 7] {MR0413144}, \cite[Section 2]{MR809499}):
\begin{enumerate}
 \item\label{prop 1} $(q_*\circ tr)(a)=|G|a$.
 \item\label{prop 2} $(tr\circ q_*)(x)=\sum_{g\in G}g_*(x)$.
 \item\label{prop 3} $q_*\colon H_i(X;R)^G\to H_i(X/G;R)$ is an isomorphism if $|G|$ is invertible in $R$.
\end{enumerate}
Here $H_i(X;R)^G:=\{x\in H_i(X;R)|\;g_*(x)=x\text{ for all }g\in G\}$.\\

Note that, provided $|G|$ is invertible in $R$, (\ref{prop 1}) and (\ref{prop 2}) together imply that $H_i(X;R)\cong ker(q_*)\oplus im(tr)=ker(q_*)\oplus H_i(X;R)^G$ and that $q_*$ is surjective. Hence, in particular, property (\ref{prop 3}) follows: Define $s\coloneqq tr\circ\cdot\frac{1}{|G|}$. Property (\ref{prop 1}) gives 
\begin{equation*}
    q_*\circ s=id_{H_*(X/G;R)}.
\end{equation*}
Thus $q_*$ is surjective, $tr$ is injective, $im(s)=im (tr)$  and $H_i(X;R)\cong ker(q_*)\oplus im(tr)$. Restricted to $H_i(X;R)^G$, $s$ is also a left-inverse for $q_*$ by property (\ref{prop 2}): For $x\in H_i(X;R)^G$ we have
\begin{equation*}
    tr\circ q_*(x)=\sum_{g\in G}g_*(x)=|G|x,
\end{equation*}
i.e. $s\circ q_*(x)=x$, as $tr$ and $q_*$ are linear. Therefore $H_i(X;R)^G\subset im(s)=im (tr)$. Finally, the two properties together imply that $H_i(X;R)^G=im (tr)$: If $x=s(a)$, then $q_*(x)=q_*\circ s(a)=a$ by the first property and hence $s\circ q_*(x)=x$. Property (\ref{prop 2}) now gives
\begin{equation*}
    x=s\circ q_*(x)=tr\circ q_*\big(\tfrac{x}{|G|}\big)=\sum_{g\in G}g_*\big(\tfrac{x}{|G|}\big),
\end{equation*}
which is certainly invariant under each $g_*$ for all $g\in G$.

Furthermore, in this case, i.e. when $|G|$ is invertible in $R$, the transfer maps are uniquely determined by the Properties \ref{prop 1}-\ref{prop 3}: For each $a\in H_i(X/G;R)$ there is a $x\in H_i(X;R)^G$ with $q_*(x)=a$, hence if $tr$ and $\wt{tr}$ are two transfer maps, then $tr(a)-\wt{tr}(a)=tr\left(q_*(x)\right)-\wt{tr}\left(q_*(x)\right)=\sum_{g\in G}g_*(x)-\sum_{g\in G}g_*(x)=0$. In all of our applications $|G|$ will be invertible in $R$, in fact we will choose $R=\Q$. Moreover, attention will only be given to $G$-actions for which $\Lambda$ and $\Lambda/G$ have the homotopy types of $CW$-complexes. Hence we can use either of the definitions of the transfer referred to above. Nevertheless, we are going to use the transfer definition given in \cite{MR809499}, since it also covers the cases of $G$-actions for which $\Lambda/G$ might not have the homotopy type of a $CW$-complex.

\subsection{The transfer product}
\begin{dfn} Let $M$ be a compact, connected, oriented manifold. Let $G$ be any finite, discrete topological group acting continuously on $\Lambda=\Lambda M$. Then the "transfer product"
\begin{equation*}
P_G\colon H_i(\Lambda/G;R)\times H_j(\Lambda/G;R)\to H_{i+j-n}(\Lambda/G;R) 
\end{equation*} 
on singular homology with arbitrary coefficients $R$ is defined by
\begin{equation}\label{transfer product}
 P_G(a,b)\coloneqq q_*\big(tr(a)\ast tr(b)\big)
\end{equation} 
for classes $a,b\in H_*(\Lambda/G;R)$, where $\ast$ is the Chas-Sullivan product.
\end{dfn}

The group actions of our interest will induce Chas-Sullivan algebra isomorphisms, hence the following lemma will be useful: 

\begin{lem}
For arbitrary coefficients $R$, if $g_*\colon H_i(\Lambda M;R)\to H_i(\Lambda M;R)$ is a Chas-Sullivan algebra isomorphism for all $g\in G$, then $H_*(\Lambda;R)^G$ is invariant under the Chas-Sullivan product. Hence we can write $\big(H_*(\Lambda;R)^G,\ast)$ to denote the subalgebra of $G$-invariant classes.
\end{lem}

\begin{proof}
Let $x,y$ be elements of $H_*(\Lambda;R)^G$, then $g_*(x\ast y)=g_*(x)\ast g_*(y)=x\ast y$ for all $g\in G$.
\end{proof}

Note that, in general, the submodule $ker(q_*)$ is not invariant under an algebra action of $G$ on $H_*(\Lambda;R)$.

\begin{thm}\label{alg iso}
Assume that $|G|$ is invertible in $R$ and that $G$ induces Chas-Sullivan algebra isomorphisms. Then $\big(H_*(\Lambda/G;R),P_G\big)$ is up to scaling isomorphic to the restriction of the Chas-Sullivan product to the $G-$invariant classes. More precisely,
\begin{align*}
  q_*\colon\big(H_*(\Lambda;R)^G,\ast\big)&\xrightarrow{\cong} \big(H_*(\Lambda/G;R),P_G\big)
\end{align*}
and
\begin{align*}
  tr\colon\big(H_*(\Lambda/G;R),P_G\big)&\xrightarrow{\cong}\big(H_*(\Lambda;R)^G,\ast\big)
\end{align*}
are algebra isomorphisms up to scaling, satisfying
\begin{align*}
  q_*(x\ast y)&=\frac{1}{|G|^2}P_G\big(q_*(x),q_*(y)\big),\\
  tr\big(P_G(a,b)\big)&=|G|tr(a)\ast tr(b).
\end{align*}
In particular, $\big(H_*(\Lambda/G;R),P_G\big)$ is an associative and graded-commutative algebra.
\end{thm}

\begin{proof}
Let $x,y$ be elements of $H_*(\Lambda;R)^G$, then, by property (\ref{prop 2}), 
\begin{align*}
    P_G\big(q_*(x),q_*(y)\big)&=q_*\big(tr\circ q_*(x)\ast tr\circ q_*(x)\big)=q_*\Big(\big(\sum_{g\in G}g_*(x)\big)\ast \big(\sum_{h\in G}h_*(y)\big)\Big)\\
    &=q_*(|G|x\ast|G|y)=|G|^2q_*(x\ast y)\,.
\end{align*}
So $q_*$ is in fact algebra homomorphism for arbitrary coefficients $R$. If $|G|$ is invertible in $R$ then $im(tr)=H_*(\Lambda;R)^G$ and hence
\begin{align*}
   tr\big(P_G(a,b)\big)&=tr\circ q_*\big(tr(a)\ast tr(b)\big)=\sum_{g\in G}g_*\left(tr(a)\ast tr(b)\right)=\sum_{g\in G}g_*\left(tr(a)\right)\ast g_*\left(tr(b)\right)\\
   &=\sum_{g\in G}tr(a)\ast tr(b)=|G|tr(a)\ast tr(b)
\end{align*}
for $a,b\in H_*(\Lambda/G;R)$
by property (\ref{prop 2}) and the previous lemma.

That both homomorphisms are isomorphisms if $|G|$ is invertible in $R$ follows since Property (\ref{prop 3}) says that $q_*\colon H_i(\Lambda;R)^G\to H_i(\Lambda/G;R)$ and $tr\colon H_i(\Lambda/G;R)\to H_i(\Lambda;R)^G$ are module isomorphisms for all $i$ in this case.
\end{proof}

\begin{prp}
Assume that $g\colon\Lambda\to \Lambda$ is the identity on the point curves (= constant loops) for all $g\in G$. Also assume that $|G|$ is invertible in $R$ and that $G$ induces Chas-Sullivan algebra isomorphisms. Then the pair $\big(H_*(\Lambda/G;R),P_G\big)$ is also a unital algebra.
\end{prp}

\begin{proof}
We only need to show that the unit $E$ of $\big(H_*(\Lambda;R),\ast)$ is an element of $\big(H_*(\Lambda;R)^G,\ast)$. Since $M$ is a compact, connected, oriented manifold of dimension $n$, there is a unit $E\in H_n(\Lambda;R)$ for the Chas-Sullivan product. We define
\begin{equation*}
 e\coloneqq\frac{1}{|G|^2}q_*(E).
\end{equation*}
This in nonzero, since $E$ is the orientation class of the base manifold $M$, which is equivariantly embedded into $\Lambda M$ via the point curves, and so $g_*(E)=E$ for all $g\in G$ by assumption. For an arbitrary $a\in H_i(\Lambda/G;R)$ there is a unique $x\in H_i(\Lambda;R)^G$ with $q_*(x)=a$. We compute, using the properties of the transfer, that
\begin{align*}
 P_G(a,e)&=\frac{1}{|G|^2}q_*\big((tr\circ q_*)(x)\ast (tr\circ q_*)(E)\big)=\frac{1}{|G|^2}q_*\big(|G|x\ast|G|E\big)=q_*(x)=a.
\end{align*}
Likewise $P_G(e,a)=a$.
\end{proof}
We thus have proved: 
\begin{thm}\label{rational algebra}
Let $M$ be a compact, connected, oriented manifold and $\Lambda=\Lambda M$ its free loop space. Let a finite discrete group $G$ act continuously on $\Lambda$. Assume that $G$ induces Chas-Sullivan algebra isomorphisms which are the identity on point curves. Then $\big(H_*(\Lambda/G;\Q),P_G\big)$ is an associative, graded-commutative unital $\Q$-algebra.\qed
\end{thm}

In particular, for $n\geq3$ the algebras $\big(H_*(\Lambda S^n/\vartheta;\Q),P_\vartheta\big)$ and $\big(H_*(\Lambda S^n/\theta;\Q),P_\theta\big)$ are associative, graded-commutative and unital, as we will see in the next section.\\

What comes next is satisfyed by the action $\vartheta$ but not by $\theta$ since $\vartheta$ preserves base points while $\theta$ does not: 

\begin{prp}\label{ev/G homo}
 In the situation of the above theorem, if all $g$ in addition leave base points fixed, i.e. $g(\gamma)(0)=\gamma(0)$ for all $\gamma\in\Lambda M$ and all $g\in G$, then
 \begin{equation*}
  (ev_0/G)_*\colon\big(H_*(\Lambda/G;\Q),P_G\big)\to\big(H_*(M;\Q),\bullet)
 \end{equation*}
is an algebra homomorphism up to scaling with scaling factor $|G|^2$.
\end{prp}

\begin{proof}
If base points are fixed by $G$, then $ev_0\coloneqq\Lambda M\to M$ factors through $\Lambda M/G$. Hence, the proposition follows by Proposition \ref{alg hom CS Pont} and Theorem \ref{alg iso}.
\end{proof}

If each $g\in G$ leaves base points fixed, i.e. if $g(\gamma)(0)=\gamma(0)$ for all $\gamma\in\Lambda M$ and all $g\in G$, then we have the following commutative diagram
\begin{equation*}
\xymatrixrowsep{1pc}
\xymatrixcolsep{4pc}
\xymatrix{
\Omega\ar[r]^-j\ar[d]&\Lambda\ar[r]^-{ev_0}\ar[d]&M\ar[d]\\
\Omega/G\ar[r]\ar[r]^-{j/G}&\Lambda/G\ar[r]^-{ev_0/G}& M,
 }
\end{equation*}
where the lower row is also a fibration (since $ev_0:\Lambda\rightarrow M$ is locally trivial, the action is fibrewise and the base is paracompact). Also both rows of the diagram
\begin{equation*}
\xymatrixrowsep{1pc}
\xymatrixcolsep{4pc}
\xymatrix{
\Omega\times EG\ar[r]^-{j\times id}\ar[d]&\Lambda\times EG\ar[r]^-{ev_0\times id}\ar[d]&M\times EG\ar[d]\\
\Omega\times_G EG\ar[r]\ar[r]^-{(j\times id)/G}&\Lambda\times_G EG\ar[r]^-{(ev_0\times id)/G}& M\times BG
 }
\end{equation*}
are fibrations. If we use rational coefficients it does not matter whether we are working with quotients or homotopy quotients:

\begin{lem}
Let $G$ be a finite discrete group acting continuously on a paracompact Hausdorff space $X$ that is homotopy equivalent to a CW-complex. Then
 \begin{equation}
 H_i^{G}(X;\Q)\cong H_i(X/G;\Q).
\end{equation}
\end{lem}

\begin{proof}[Proof using transfer]
Also $EG$ is a (countable) CW-complex and we have isomorphisms given by the quotient maps (property \eqref{prop 3}):
 \begin{align*}
  H_i^{G}(X;\Q)&:=H_i(X\times_G EG;\Q)\cong H_i(X\times EG;\Q)^G,\\
  H_i(X/G;\Q)&\cong H_i(X;\Q)^G,
 \end{align*}
where again $H_i(\cdot;\Q)^G$ denotes the classes of degree $i$ that are fixed under $g_*$ for all $g\in G$. As $EG$ is contractible the K\"unneth formula gives
\begin{equation*}
 H_i(X\times EG;\Q)\cong\bigoplus_k \big(H_{i-k}(X;\Q)\otimes_\Q H_k(EG;\Q)\big)\cong H_i(X;\Q)\,.
\end{equation*}
It follows that for $y\in H_i(X\times EG;\Q)$, given a generator $1\in H_0(EG;\Q)$ we have $y=x\times 1$ for a unique $x\in H_i(X;\Q)$. Then
\begin{equation*}
 (g\times g)_*(y)=g_*(x)\times g_*(1)=g_*(x)\times 1
\end{equation*}
since any continuous map induces the identity on $H_0(EG;\Q)$ as $EG$ is path-connected. Hence, $y$ is fixed under $G$ if and only if $x$ is and so
\begin{equation*}
 H_i^{G}(X;\Q)\cong H_i(X\times EG;\Q)^G\cong H_i(X;\Q)^G\cong H_i(X/G;\Q)
\end{equation*}
for all $i\in\Z$.
\end{proof}

The above result can be used for a straightforward definition of transfer in rational homology for pairs $(X,A)$ where both $X$ and $A$ are paracompact Hausdorff and of the homotopy type of a CW-complex. For instance, for the pair $(E,E-B)$ where $E\rightarrow B$ is a (smooth) vector bundle over $B$, the following definition can be made: Just take the transfer map of the covering $E\times EG\to E\times_G EG$ and compose with the above isomorphism. Also, naturality of the transfer is easily proved in this way.

Just as in the case of $\Lambda M/G$ one can also define a transfer product on the homology of $\Omega_p/G$ for the submanifold $\Omega_p:=\{\gamma\in\Lambda M|\,\gamma(0)=p\}\subset \Lambda M$ via $q_*\circ tr \star tr$ where $\star$ is the Pontrjagin product on $H_*(\Omega_p)$. We denote it by $P_G^\Omega$.

We write $\Lambda:=\Lambda M$ in the remainder of this subsection. We now consider the inclusion $j_p:\Omega_p\hookrightarrow \Lambda$ of the loops based at a point $p\in M$. We fix some base point $p\in M$ and just write $j$ and $\Omega$ for $j_p,\Omega_p$ from now on. We again use that $ev_0:\Lambda\rightarrow M$ is a submersion (compare Lemma \ref{nb is pullback}). Thus, if we consider the inclusion of the point $\{p\}\hookrightarrow M$, it follows that $ev_0^{-1}(\{p\})=\Omega\subset\Lambda$ is a submanifold of codimension $n$ and has trivial normal bundle $N_\Omega$ isomorphic to a pullback of a (normal) bundle of a point, $N_\Omega\cong {ev_0}^*(T_pM)$. This assures the existence of a Thom class for any coefficients. Thus for arbitrary coefficients also the Gysin map $j_!\colon H_*(\Lambda)\to H_{*-n}(\Omega)$ is defined. For convenience we recall the definition of $j_!$ here: Let $t\colon N\to U$ be a tubular neighbourhood map onto an open neighbourhood $U$ of $\Omega$ in $\Lambda$. The map $j_!$ is then defined as the composition such that the diagram
\begin{equation*}
\xymatrixrowsep{1pc}
\xymatrixcolsep{4pc}
\xymatrix{
H_*(\Lambda)\ar[d]\ar[rr]^-{j_!}&&H_{*-n}(\Omega)\\
H_*(\Lambda,\Lambda-\Omega)\ar[r]_-{\text{excision}}^\cong&H_*(U,U-\Omega)\ar[r]_-{t_*^{-1}}^\cong&H_*(N_\Omega,N_\Omega-\Omega)\ar[u]_-{\text{Thom isomorphism}}^\cong
 }
\end{equation*}
commutes. The definition does not depend on the particular choice of a tubular neighbourhood. Assume now that a group $G$ acts on $\Lambda$ leaving the starting points of loops fixed, so that $j$ is a $G$-equivariant embedding and $j/G$ is defined. Moreover, since $ev_0$ is equivariant with respect to the trivial action on $M$, $G$ acts trivially on $N_\Omega$ and the Thom class passes to the quotient. If in addition also the tubular neighbourhood map $t$ is equivariant $(j/G)_!$ can be defined (see proof below).

\begin{prp}\label{j/G homo}
In the situation of Theorem \ref{rational algebra}, assume that $g\colon\Lambda M\to \Lambda M$ fixes the starting point of each loop for each $g\in G$. Also assume that $\Omega_p\subset\Lambda$ posseses a equivariant tubular neighbourhood for all $p$. Then
 \begin{equation*}
  (j/G)_!\colon\big(H_*(\Lambda M/G;\Q),P_G\big)\to\big(H_{*-n}(\Omega M/G;\Q),P_G^\Omega)
 \end{equation*}
is an algebra homomorphism. Note the index shift.
\end{prp}
The $\Z_2$-action $\vartheta$ for example satisfies the assumptions in this proposition.

\begin{proof}
We have to show that the diagram 
\begin{equation}\label{j! homo}
\xymatrixrowsep{1pc}
\xymatrixcolsep{6pc}
\xymatrix{
H_{i-n}(\Omega/G;\Q)\times H_{j-n}(\Omega/G;\Q)\ar[d]_{tr\times tr}&H_i(\Lambda/G;\Q)\times H_j(\Lambda/G;\Q)\ar[l]_-{(j/G)_!\times (j/G)_!}\ar[d]^{tr\times tr}\\
H_{i-n}(\Omega;\Q)\times H_{j-n}(\Omega;\Q)\ar[d]_\star&H_i(\Lambda;\Q)\times H_j(\Lambda;\Q)\ar[l]_-{j_!\times j_!}\ar[d]^\ast\\
H_{i+j-2n}(\Omega;\Q)\ar[d]_{q_*}&H_{i+j-n}(\Lambda;\Q)\ar[l]_-{j_!}\ar[d]^{q_*}\\
H_{i-n}(\Omega/G;\Q)&H_i(\Lambda/G;\Q)\ar[l]_-{(j/G)_!}
 }
\end{equation}
commutes. The middle square commutes by Proposition \ref{alg hom CS Pont}.

We now show that the top square commutes by proving that the maps in the definition of $j_!$ pass to quotients such that $j_!$ and $(j/G)_!$ commute with $tr$: Let $\tau_\Omega\in H^n(N_\Omega,N_\Omega-\Omega;\Q)$ denote a Thom class of the normal bundle $N_\Omega$ of $\Omega\subset\Lambda$. $\tau_\Omega$ corresponds to $pr^*(\sigma)\in H^n({ev_0}^*(T_pM),{ev_0}^*(T_pM)-\Omega;\Q)$ where $\sigma$ is some generator of $ H^n(\R^n,\R^n-0;\Q)\cong \Q$ and $pr\colon{ev_0}^*(T_pM)\to T_pM$ is the projection. Since $G$ acts on $\Lambda$ leaving the starting point of each loop fixed, for the induced action on $\Omega$ we have the commutative diagram of pullback squares
\begin{equation*}
\xymatrixrowsep{1pc}
\xymatrixcolsep{2.5pc}
 \xymatrix{
   {ev_0}^*(T_pM)\ar[r]_-{Q=q\times id}\ar[d]_-{p}\ar@/^1pc/[rr]^-{pr}&{\big(ev_0/G\big)}^*(T_pM)\ar[r]_-{pr/G}\ar[d]_-{p/G} &T_pM\cong \R^n\ar[d]\\
   \Omega\ar[r]^-q\ar@/_1pc/[rr]_-{ev_0}&\Omega/G\ar[r]^-{ev_0/G}&{\{p\}}
     }
\end{equation*}
where we view $T_pM$ as a trivial $G$-space. Hence the quotient map $Q\coloneqq q\times id$ is the orbit projection of the diagonal action on ${ev_0}^*(T_pM)$, where $q\colon\Omega\to \Omega/G$ is the orbit projection of the action on $\Omega$. As $|G|$ is invertible in $\Q$, it follows that the diagram
\begin{equation}\label{j_!/G}
\xymatrixrowsep{1pc}
\xymatrixcolsep{3pc}
\xymatrix{
H_{i-n}\big({(ev_0/G)}^*(T_pM);\Q\big)\ar[d]_{tr}&H_i\big({(ev_0/G)}^*(T_pM),{(ev_0/G)}^*(T_pM)-\Omega/G;\Q\big)\ar[d]_{tr}\ar[l]_-{\cap (pr/G)^*(\sigma)}\\
H_{i-n}\big({ev_0}^*(T_pM);\Q\big)&H_i\big({ev_0}^*(T_pM),{ev_0}^*(T_pM)-\Omega;\Q\big)\ar[l]_-{\cap pr^*(\sigma)}
 }
\end{equation}
is commutative as we show now: Given $a\in H_i\big({(ev_0/G)}^*(T_pM),{(ev_0/G)}^*(T_pM)-\Omega/G;R\big)$ we compute
\begin{align*}
    Q_*\big(tr(a)\cap pr^*(\sigma)\big)&=Q_*\Big(tr(a)\cap Q^*\big((pr/G)^*(\sigma)\big)\Big)=Q_*\big(tr(a)\big)\cap (pr/G)^*(\sigma)\\
    &=|G|a\cap (pr/G)^*(\sigma)
\end{align*}
by property (\ref{prop 1}) of the transfer and the naturality of the cap product. Hence
\begin{align*}
    |G|tr\big(a\cap &(pr/G)^*(\sigma)\big)=(tr\circ Q_*)\Big(tr(a)\cap Q^*\big((pr/G)^*(\sigma)\big)\Big)\\
    &=\sum_{g\in G}g_*\Big(tr(a)\cap Q^*\big((pr/G)^*(\sigma)\big)\Big)=\sum_{g\in G}g_*\Big(tr(a)\cap (Q\circ g)^*\big((pr/G)^*(\sigma)\big)\Big)\\
    &=\sum_{g\in G}g_*\big(tr(a)\big)\cap Q^*\big((pr/G)^*(\sigma)\big)=\sum_{g\in G}(tr(a)\cap Q^*\big((pr/G)^*(\sigma)\big)\\
    &=|G|\big(tr(a)\cap Q^*\big((pr/G)^*(\sigma)\big)
    =|G|\big(tr(a)\cap pr^*(\sigma)\big)
\end{align*}
by property (\ref{prop 2}) of the transfer, naturality of the cap product and since $im(tr)$ is invariant under $G$. Since multiplication with $|G|$ is an isomorphism
\begin{equation*}
    tr\big(a\cap (pr/G)^*(\sigma)\big)=tr(a)\cap pr^*(\sigma)
\end{equation*}
follows, i.e. diagram \eqref{j_!/G} is commutative. (Notice that in our particular situation we could have shown the above equation by using that ${ev_0}^*(T_pM)=\Omega\times T_pM$, $pr^*(\sigma)=1\times\sigma$ and analogously for ${(ev_0/G)}^*(T_pM)$. Then one can just take the transfer on the factor $\Omega$.)

On the left we extend diagram \eqref{j_!/G} with
\begin{equation*}
\xymatrixrowsep{1pc}
\xymatrixcolsep{3pc}
\xymatrix{
H_{i-n}(\Omega/G;\Q)\ar[d]_{tr_q}&H_{i-n}\big({(ev_0/G)}^*(T_pM);\Q\big)\ar[d]^{tr_Q}\ar[l]_-{(p/G)_*}\\
H_{i-n}(\Omega;\Q)&H_{i-n}\big({ev_0}^*(T_pM);\Q\big)\ar[l]_-{p_*}
 }
\end{equation*}
where the subscripts $q,Q$ in $tr_q$ and $tr_Q$ emphasize to which orbit map the transfer is associated.
This diagram commutes since, as $|G|$ is invertible, we have $a=Q_*(x)$ and hence
\begin{align*}
    tr_q\circ (p/G)_*(a)&=tr_q\circ (p/G)_*\big(Q_*(x)\big)=tr_q\circ q_*\circ p_*(x)=\sum_{g\in G}g_*\big(p_*(x)\big)\\
    &=\sum_{g\in G}p_*\big(g_*(x)\big)=p_*\big(\sum_{g\in G}g_*(x)\big)=p_*\circ tr_Q\circ Q_*(x)=p_*\circ tr_Q(a),
\end{align*}
since $p$ is equivariant. Abbreviating ${ev_0}^*(T_pM)$ with $T$, on the right we extend \eqref{j_!/G} with
\begin{equation*}
\xymatrixrowsep{1pc}
\xymatrixcolsep{1pc}
\xymatrix{
H_i(T/G,T/G-\Omega/G)\ar[d]_{tr}\ar[r]^-{(t/G)_*}_-\cong&H_i(U/G,U/G-\Omega/G)\ar[d]_{tr}\ar[r]_-\cong&H_i(\Lambda/G,\Lambda/G-\Omega/G)\ar[d]_{tr}&H_i(\Lambda/G)\ar[l]\ar[d]_{tr}\\
H_i(T,T-\Omega)\ar[r]_-{t_*}^-\cong&H_i(U,U-\Omega)\ar[r]^-\cong&H_i(\Lambda,\Lambda-\Omega)&H_i(\Lambda)\ar[l]
}
\end{equation*}
(coefficients $\Q$). Here $t\colon T\to U$ is an equivariant tubular neighbourhood map, for example, for the action $\vartheta$ we could take $t(\gamma,v)\coloneqq\lambda(\gamma,v)$ where  $\lambda$ is defined as in section \ref{explicit tubes}. All the other maps involved are also equivariant and we can argue as above that the diagram commutes. The composition of the three diagrams is the commutative diagram
\begin{equation*}
\xymatrixrowsep{1pc}
\xymatrixcolsep{4pc}
\xymatrix{
H_{i-n}(\Omega/G;\Q)\ar[d]_{tr}&H_i(\Lambda/G;\Q)\ar[l]_-{(j/G)_!}\ar[d]^{tr}\\
H_{i-n}(\Omega;\Q)&H_i(\Lambda;\Q)\ar[l]_-{j_!}\,,
 }
\end{equation*}
which shows that the top square of diagram \eqref{j! homo} at the beginning of the proof is commutative. In the same way as above one shows that 
\begin{equation*}
    Q_*(x)\cap (pr/G)^*(\sigma)=Q_*\big(x\cap pr^*(\sigma)\big)
\end{equation*}
holds, from which it follows that also the bottom square of \eqref{j! homo} commutes and hence the whole diagram \eqref{j! homo} is commutative.
\end{proof}

The commutative diagram
\begin{equation*}
\xymatrixrowsep{1pc}
\xymatrix{
|G|j_!\big(tr(a)\big)\star j_!\big(tr(b)\big)&|G|tr(a)\ast tr(b) \ar@{|->}[l]_-{j_!} \ar@{|->}[r]^-{(ev_0)_*}& |G|(ev_0)_*\big(tr(a)\big)\bullet (ev_0)_*\big(tr(b)\big)\\
P_G^\Omega\big((j/G)_!(a),(j/G)_!(b)\big) \ar@{|->}[u]^-{tr}& P_G(a,b) \ar@{|->}[l]^-{(j/G)_!} \ar@{|->}[u]_-{tr} \ar@{|->}[r]_-{(ev_0/G)_*}&|G|^2 (ev_0)_*\big(tr(a)\big)\bullet (ev_0)_*\big(tr(b)\big)\ar@{|->}[u]_-{\cdot|G|^{-1}}^-\cong
 }
\end{equation*}
summarizes the situation. We observe that the left-hand square and property \eqref{prop 1} show that $(j/G)_!=\frac{1}{|G|}q_\Omega\circ j_!\circ tr_\Lambda$ and the right-hand square shows that we could define $(ev_0/G)_*\coloneqq|G|(ev_0)_*\circ tr_\Lambda$ even for actions of $G$ that do not leave base points fixed. Here the subscripts $\Lambda$ and $\Omega$ indicate to the $G$-action on which space the maps are associated to.

\subsection{Group actions under consideration: $G$ a finite subgroup of $O(2)$}\label{our actions}
For a Riemannian manifold $(M,g)$, $O(2)$ acts continuously on $(\Lambda M,g_1)$ via isometries in the following way: The $O(2)\cong S^1\rtimes\Z_2$-action on $S^1$ is generated by rotations and a reflection. Viewing the circle as parametrized from zero to one, we like to denote the induced maps on 
    \begin{align*}
 \Lambda M\cong\Big\{\gamma:[0,1]\to M\Big|\,\gamma(0)=\gamma(1),\gamma\text{ is absolutely continuous and }\\\int_0^1g\big(\dot{\gamma}(t),\dot{\gamma}(t)\big)\mathrm{d}t<\infty\Big\}
\end{align*}
by
\begin{itemize}
    \item $\chi_s\colon\Lambda M\to\Lambda M,\,\chi_s(\gamma)(t)\coloneqq\gamma(t+s)$,
    \item $\vartheta\colon\Lambda M\to\Lambda M,\,\vartheta(\gamma)(t)\coloneqq\gamma(1-t)$.
\end{itemize}
Note that $\chi_s\circ \vartheta=\vartheta\circ\chi_{-s}$. Of special interest to us is also the involution
\begin{equation*}
    \theta\coloneqq \vartheta\circ\chi_{\frac{1}{2}}=\chi_{\frac{1}{2}}\circ\vartheta
\end{equation*}
mentioned in the introduction. It reverses the orientation and shifts the starting point about $1/2$: $\theta(\gamma)(t)\coloneqq\gamma\big(1-(t+\frac{1}{2})\big)$.

The involutions $\vartheta,\theta$ and the rotations $\chi_s$ all leave the energy $E_g$ of every metric $g$ on $M$ invariant and are isometries of $(\Lambda M,g_1)$ (Lemma 2.2.1 and Theorem 2.2.5 in \cite{MR0478069}).

For rational coefficients, there is a relation between the products $P_\vartheta$ and $P_\theta$. The geometric differences between the actions do not seem to be relevant when $2$ is invertible: The diffeomorphism $\chi_{\frac{1}{4}}\colon\Lambda\to\Lambda$ is equivariant with respect to the $\Z_2$-actions induced by $\vartheta$ and $\theta$:
\begin{equation*}
\vartheta=\vartheta\circ id=\vartheta\circ\chi_{\frac{1}{2}}\circ\chi_{\frac{1}{2}} =\theta\circ\chi_{\frac{1}{2}} =\theta\circ\chi_{\frac{1}{4}}\circ\chi_{\frac{1}{4}}
\end{equation*}
and hence $\chi_{\frac{1}{4}}\circ\vartheta=\vartheta\circ(\chi_{\frac{1}{4}})^{-1}=\theta\circ\chi_{\frac{1}{4}}$. Let $\chi:=\chi_{\frac{1}{4}}/\Z_2:\Lambda/\vartheta\cong\Lambda/\theta$ denote the homeomorphism induced by $\chi_{\frac{1}{4}}$.

\begin{prp} \label{algebra iso}
$\chi_*\colon H_*(\Lambda/\vartheta;\Q) \to H_*(\Lambda/\theta;\Q)$ is an algebra isomorphism between $\big(H_*(\Lambda/\vartheta;\Q),P_\vartheta\big)$ and $\big(H_*(\Lambda/\theta;\Q),P_\theta\big)$.
\end{prp}

\begin{proof}
Let $a,b\in H_*(\Lambda/\vartheta)$. If we consider homology with coefficients in $\Q$, the quotient maps $q_\theta,q_\vartheta$ of the two actions induce surjective maps in homology. Thus there are uniquely determined classes $x,y\in H_*(\Lambda;\Q)^\vartheta=H_*(\Lambda;\Q)^\theta$ with ${q_\vartheta}_*(x)=a,\,{q_\vartheta}_*(y)=b$. It follows that 
\begin{align*}
tr_\theta\big(\chi_*(a)\big)&=tr_\theta\Big(\chi_*\big({q_\vartheta}_*(x)\big)\Big)=tr_\theta\Big({q_\theta}_*\big({\chi_{\frac{1}{4}}}_*(x)\big)\Big)=tr_\theta\big({q_\theta}_*(x)\big),
\end{align*}
since $\chi_{\frac{1}{4}} $ is homotopic to the identity. Using the properties of the transfer homomorphisms we then get
\begin{equation*}
tr_\theta\big({q_\theta}_*(x)\big)=x+\theta_*(x)=2x=x+\vartheta_*(x)=tr_\vartheta\big({q_\vartheta}_*(x)\big)=tr_\vartheta(a),
\end{equation*}
as $\vartheta_*=\theta_*$. Therefore 
 \begin{align*}
  P_\theta\big(\chi_*(a),\chi_*(b)\big)&={q_\theta}_*\Big(tr_\theta\big(\chi_*(a)\big)\ast tr_\theta\big(\chi_*(b)\big)\Big)={q_\theta}_*\big(tr_\vartheta(a)\ast tr_\vartheta(b)\big)\\
  &={q_\theta}_*\circ{\chi_{\frac{1}{4}}}_*\big(tr_\vartheta(a)\ast tr_\vartheta(b)\big)=(\chi_*\circ{q_\vartheta}_*)\big(tr_\vartheta(a)\ast tr_\vartheta(b)\big)\\
  &=\chi_*\big(P_\vartheta(a,b)\big)
 \end{align*}
 and $\chi_*$ is an algebra homomorphism. It is an isomorphism since it inverse is also an algebra homomorphism as can be seen in the same way.
 \end{proof}
 
\begin{cor}
 The homomorphism 
  \begin{equation*}
  (ev_0/\vartheta)_*\colon\big(H_*(\Lambda/\vartheta;\Q),P_\vartheta\big)\to\big(H_*(M;\Q),\bullet)
 \end{equation*}
 and 
  \begin{equation*}
  (ev_0/\theta)_*\circ(\chi_*)^{-1}\colon\big(H_*(\Lambda/\theta;\Q),P_\theta\big)\to\big(H_*(M;\Q),\bullet)
 \end{equation*}
 are algebra homomorphisms up to scaling.
\end{cor}

The discrete subgroups of $O(2)\cong S^1\rtimes\Z_2$ are isomorphic to the following finite groups:
\begin{itemize}
\item the groups $C_m$ given by the inclusions $\Z_m\hookrightarrow S^1\subset O(2),\,[n]\mapsto e^{2\pi i\frac{n}{m}}$.
\item the dihedral groups $D_m\coloneqq C_m\rtimes\Z_2$.
\end{itemize}

\begin{lem}
For arbitrary coefficients $R$, all $i\in \N$ and all $m\in \N$, we have
\begin{itemize}
    \item $H_i(\Lambda M;R)^{C_m}=H_i(\Lambda M;R)$\,,
    \item $H_i(\Lambda M;R)^{D_m}=H_i(\Lambda M;R)^\vartheta$\,.
\end{itemize}
Furthermore
\begin{equation*}
    H_i(\Lambda M;R)^\vartheta =H_i(\Lambda M;R)^{D_1}=H_i(\Lambda M;R)^{\chi_\frac{1}{4}D_1\chi_\frac{1}{4}^{-1}}=H_i(\Lambda M;R)^\theta
\end{equation*}
\end{lem}

\begin{proof}
\begin{itemize}
    \item Each element $\chi$ of $C_m$ is a rotation and hence $\chi_*=id_*=id:H_i(\Lambda M)\to H_i(\Lambda M)$.
    \item $D_m$ acts as the subgroup $\{id, \chi_\frac{1}{n},\dots,\chi_\frac{n-1}{n}, \vartheta,\chi_\frac{1}{n}\circ\vartheta,\dots,\chi_\frac{n-1}{n}\circ\vartheta\}$ of the diffeomorphism group of $\Lambda M$. It follows that a homology class $x$ is invariant under the induced action of $D_m$ if and only if it is invariant under $\vartheta_*:H_i(\Lambda M)\rightarrow H_i(\Lambda M)$.
\end{itemize}
The last assertion follows immediately.
\end{proof}
Together with property (\ref{prop 3}) of transfers this yields 

\begin{prp}\label{reversible vs. nonreversible?}
\begin{itemize}
\item For the action of $C_m$ on $\Lambda M$ we have, for all $i\in\Z$ and all $m\in\N$,
\begin{equation*}
    H_i(\Lambda M;\Q)^{C_m}=H_i(\Lambda M;\Q)\cong H_i(\Lambda M/C_m;\Q)
\end{equation*}
as $\Q$-modules via $q_*$ or $tr$.
\item For the action of $D_m$ on $\Lambda M$ we have, for all $i\in\Z$ and all $m\in\N$,
\begin{equation*}
    H_i(\Lambda M;\Q)^{D_m}=H_i(\Lambda M;\Q)^\vartheta\cong H_i(\Lambda M/D_m;\Q)
\end{equation*}
as $\Q$-modules via $q_*$ or $tr$.
\end{itemize}
In particular, this means that taking the quotient with respect to $C_m$ has no effect on rational homology.\qed
\end{prp}

Due to its "geometric definition" the Chas-Sullivan product behaves well with respect to the $O(2)$-action:

\begin{prp}\label{inv vs mult CS}
Let $M$ be an $n$-dimensional connected, compact, smooth manifold that is $R$-orientable. Then, the homomorphism $\vartheta_*=\theta_*:\big(H_*(\Lambda M;R),\ast)\to \big(H_*(\Lambda M:R),\ast)$ is an algebra endomorphism: for all $a\in H_i(\Lambda;R),b\in H_j(\Lambda;R)$ we have
  \begin{equation}\label{alg hom}
 \vartheta_*(a)\ast\vartheta_*(b)=\vartheta_*(a\ast b).	
 \end{equation}
 Hence, our $O(2)$-action induces an action by algebra isomorphims on $(H_*(\Lambda M;R),\ast)$.
\end{prp}

\begin{proof}
Since our definition of the Chas-Sullivan involves capping with a Thom class $\tau_\mathcal{F}$ of $\mathcal{F}\subset\Lambda^2\coloneqq \Lambda\times\Lambda$, we first show that $\tau_\mathcal{F}$ is invariant under $\vartheta\times\vartheta$: Let $N$ denote the normal bundle of the inclusion $M\cong\Delta(M)\subset M\times M=:M^2$.
Recall from Lemma \ref{nb is pullback} that the normal bundle $N_\mathcal{F}$ of the figure-eight space $\mathcal{F}\subset \Lambda^2$ is isomorphic to ${ev_0}^*(N)$, where $ev_0\coloneqq\mathcal{F}\to M,\,\left(\gamma,\delta\right)\to\gamma(0)$. $ev_0$ is just the restriction of $ev_0\times ev_0\colon\Lambda ^2\to M^2$ to $\mathcal{F}\subset \Lambda^2$ and $M\cong\Delta(M)\subset M^2$. A Thom class $\tau_{\mathcal{F}}\in H^n(N_{\mathcal{F}},N_{\mathcal{F}}-\mathcal{F})$ of $N_\mathcal{F}$ is then a pullback of a Thom class $\tau\in H^n(N,N-M;R)$ of $N$ via the differential of $ev_0\times ev_0$ restricted to normal directions. Since the diagram
\begin{equation*}
\xymatrixrowsep{1pc}
    \xymatrix{
\Lambda^2\ar[rd]_-{ev_0\times ev_0}\ar[rr]^-{\vartheta\times\vartheta}&&\Lambda^2\ar[ld]^-{ev_0\times ev_0}\\
&M^2&
 }
\end{equation*}
commutes, it follows that $(\vartheta\times\vartheta)^*(\tau_\mathcal{F})=\tau_\mathcal{F}$. Furthermore, there is a tubular neighbourhood map that commutes with $\vartheta\times\vartheta$: Recall the map $h$ that pushes points from the section on tubuluar neighbourhoods (Section \ref{explicit tubes}). As
\begin{align*}
\vartheta\big(\lambda(\delta,v)\big)(t)&=h\big(\delta(0),\exp(\delta(0),v)\big)\big(\delta(1-t)\big)=h\big(\delta(0),\exp(\delta(0),v)\big)\big(\vartheta(\delta)(t)\big)\\
&=\lambda\big(\vartheta(\delta),v\big)(t),
\end{align*} 
the explicit tubular neighbourhood map $t=t_\mathcal{F}:{ev_0}^*(N)\rightarrow \Lambda^2$ defined in Section \ref{explicit tubes} satisfies  $(\vartheta\times\vartheta)\circ t=t\circ\left((\vartheta\times\vartheta)\times id_N\right)$, where $(\vartheta\times\vartheta)\times id_N$ corresponds to the differential of $\vartheta\times\vartheta$ in ${ev_0}^*(N)\cong N_\mathcal{F}$. Thus, mildly abusing notation by also using the symbol $\tau_\mathcal{F}$ for the Thom class understood as an element of $H^n(\Lambda^2,\Lambda^2-\mathcal{F})$, we hence have the following commutative diagram ($R$-coefficients), where the horizontal compositions precomposed with the cross product are the Chas-Sullivan product:
\begin{equation*}
 \xymatrixcolsep{6pc}\xymatrix{
   H_{i+j}(\Lambda^2,\Lambda^2-\mathcal{F})\ar[r]_-{\cap(\vartheta\times\vartheta)^*(\tau_\mathcal{F})=\tau_\mathcal{F}}\ar[d]_-{(\vartheta\times\vartheta)_*}&H_{i+j-n}(\mathcal{F})\ar[d]_-{(\vartheta\times\vartheta)_*}\ar[r]_-{\phi_*}&H_{i+j-n}(\Lambda)\ar[d]_-{\theta_*}\\ 
   H_{i+j}(\Lambda^2,\Lambda^2-\mathcal{F})\ar[r]_-{\cap\tau_\mathcal{F}}&H_{i+j-n}(\mathcal{F})\ar[r]_-{\phi_*}&H_{i+j-n}(\Lambda).
     }  
\end{equation*}
This diagram commutes since the concatenation $\phi=\phi_{\frac{1}{2}}$ at time $\frac{1}{2}$ satisfies $\phi\circ(\vartheta\times\vartheta)=\theta\circ\phi$. Recalling that for a loop $\gamma$ we have $\bar{\gamma}\coloneqq\vartheta(\gamma)$ and that $\theta=\chi_\frac{1}{2}\circ\vartheta$, that is $\theta\left(\gamma\cdot\delta\right)=\bar{\gamma}\cdot\bar{\delta}$ for loops $\gamma,\delta$. The diagram thus implies that
  \begin{equation*}
 \vartheta_*(a)\ast\vartheta_*(b)=\theta_*(a\ast b)
 \end{equation*}
holds for $a\in H_i(\Lambda;R),b\in H_j(\Lambda;R)$. Via $\vartheta_*=\theta_*$ this yields equation \eqref{alg hom}.
As any element of $O(2)$ is either a rotation $\chi$ or of the form $\chi\circ\vartheta$ the second assertion of the proposition is immediate.
\end{proof}

This together with Theorem \ref{alg iso} proves

\begin{thm}\label{O(2) subgroup iso}
Let $M$ be an $n$-dimensional connected, compact, orientable smooth manifold and $G$ a finite subgroup of $O(2)$ acting on $\Lambda M$ by linear reparametrization as described above. Then $\big(H_*(\Lambda M/G;\Q),P_G\big)$ is up to scaling isomorphic to the restriction of the Chas-Sullivan product to the $G$-invariant classes. More precisely:
\begin{itemize}
    \item $\big(H_*(\Lambda M/C_m;\Q),P_{C_m}\big)$ is up to scaling isomorphic to the entire Chas-Sullivan algebra $\big(H_*(\Lambda M;\Q),\ast\big)$.
    \item For groups $G$ conjugate to some $D_m$, $\big(H_*(\Lambda M/G;\Q),P_G\big)$ is up to scaling isomorphic to the restricted Chas-Sullivan algebra $\big(H_*(\Lambda M;\Q)^\vartheta,\ast\big)$.
\end{itemize}
In particular, $\big(H_*(\Lambda M/G;\Q),P_G\big)$ is an associative, graded-commutative and unital algebra.
\end{thm}

For groups $G$ conjugate to some $D_m$, we are going to compute the algebras $\big(H_*(\Lambda S^n/G;\Q),P_G\big)$ more explicitly (see Theorem \ref{main thm}) in the next section.

\section{The transfer product $P_\vartheta$ on $H_*(\Lambda S^n/\vartheta;\Q)$}

In this section we set $\Lambda\coloneqq\Lambda S^n$.

We now compute the transfer product $P_\vartheta$ associated to orientation reversal of loops on spheres with rational coefficients.

We first compute the rational homology of $\Lambda/\vartheta$ as follows:

\begin{itemize}
 \item We know that the surjective homomorphism $q_*$ induced by $q\colon\Lambda\to \Lambda/\vartheta$ maps the subspace of  $H_*(\Lambda;\Q)$ consisting of those classes which are fixed under $\vartheta_*$ isomorphically onto $H_*(\Lambda/\vartheta;\Q)$, i.e, for $0\neq x\in H_i(\Lambda;\Q)$, $0\neq q_*(x)\in H_i(\Lambda/\vartheta;\Q)$ if and only if $\vartheta_*(x)=x$.
 
\item We are going to see how $\vartheta_*$ acts on the generators of the rational loop homology algebra of spheres. The universal coefficient theorem for homology shows that $H_i(\Lambda S^n;\Q)\cong H_i(\Lambda S^n;\Z)\otimes\Q$ is either zero or isomorphic to $\Q$ since $H_i(\Lambda S^n;\Z)$ is either $0$ or $\Z$ or $\Z_2$. The Chas-Sullivan product is defined in exact the same way as for $\Z$-coefficients and all classes of $H_*(\Lambda S^n;\Q)$ can be written as a Chas-Sullivan product of the rational classes corresponding to the generators introduced in Section \ref{CS spheres} (compare \cite[Lemma 5.4]{hingston2013}).

\item Since we also know that $\vartheta_*(x\ast y)=\vartheta_*(x)\ast \vartheta_*(y)$ for all $x,y\in H_*(\Lambda S^n;\Q)$ (Proposition \ref{inv vs mult CS}), it suffices to know how $\vartheta_*$ acts on the generators to compute $H_*(\Lambda/\vartheta;\Q)$.\\
\end{itemize}

We compute the above relations using the Pontrjagin product on the based loop space: Let $p\in S^n$ and consider the space $\Omega_p S^n=\Omega S^n$ of loops based at $p$. Since $\Omega S^n$ is an H-space, even an H-group, its singular homology carries a Pontrjagin product. Let $\star$ denote this Pontrjagin product. It is well known that for $n\geq2$ (with $\Z$-coefficients) the Pontrjagin ring of spheres is a polynomial ring: $(H_*(\Omega S^n;\Z),\star)\cong\Z[x]$ with $x\in H_{n-1}(\Omega S^n;\Z)$ (see e.g. \cite[Section 4J]{HatcherAT}).

\begin{lem} Let $\vartheta$ be the orientation reversal of loops on the based loop space $\Omega S^n$ of $S^n$ with $n\geq2$. Then, for the generator $x\in H_{n-1}(\Omega S^n;\Z)$ of the Pontrjagin ring $(H_*(\Omega S^n;\Z),\star)\cong\Z[x]$, we have
\begin{equation*}
 \vartheta_*(x)=-x\,,
\end{equation*}
as $\vartheta$ induces $-id$ on $H_{n-1}(\Omega S^n;\Z)$.
\end{lem}

\begin{proof}
Let $\phi$ denote the concatenation on $\Omega S^n$, i.e. the H-group multiplication of $\Omega S^n$. Consider the diagonal embedding $\Delta\colon \Omega S^n\to\Omega S^n\times \Omega S^n,\,f\mapsto (f,f)$. Since $\vartheta\colon\Omega S^n\to\Omega S^n$ is the homotopy inversion of the H-group $\Omega S^n$ it follows that the composition $\phi\circ(id\times\vartheta)\circ\Delta\colon \Omega S^n\to\Omega S^n$ is homotopic to the constant map and therefore induces the zero map on $H_{n-1}(\Omega S^n;\Z)$, as $n-1\geq1$.

Let $y$ be any element of $H_{n-1}(\Omega S^n;\Z)$. For the remainder of the proof we abbreviate $\Omega S^n$ with $\Omega$ and $\Omega\times\Omega$ with $\Omega^2$ and omit the coefficients $\Z$ from the notation. By the Künneth theorem $\Delta_*(y)=a(x\times1)+b(1\times x)\in H_{n-1}(\Omega^2)$ for some integers $a,b$, since $H_{n-1}(\Omega^2)$ is isomorphic to $\bigoplus_{i+j=n-1}H_i(\Omega)\otimes H_j(\Omega )=H_{n-1}(\Omega)\otimes H_0(\Omega)\bigoplus H_0(\Omega)\otimes H_{n-1}(\Omega)$ via the cross product $\times$. Here $1$ denotes a generator of $H_0(\Omega)\cong \Z$, the unit for the Pontrjagin product. Consider the swap map $\mathrm{T}\colon\Omega^2\to\Omega^2, (\alpha,\beta)\mapsto (\beta,\alpha)$. We have $\mathrm{T}_*(x\times1)=1\times x$ and $a(x\times1)+b(1\times x)=\Delta_*(y)=\mathrm{T}_*\circ\Delta_*(y)= b(x\times1)+a(1\times x)$, which shows that $a=b$. We thus have $\Delta_*(y)=a(x\times1+1\times x)=a\left(\times(x\otimes 1+1\otimes x)\right) $. The commutative diagram
\begin{equation*}
\xymatrixrowsep{1pc}
\xymatrixcolsep{1pc}
    \xymatrix{
    H_{n-1}(\Omega)\ar@/^2pc/[rrr]^-0\ar[r]^-{\Delta_*}&H_{n-1}(\Omega^2)\ar[r]^-{(id\times\vartheta)_*}&H_{n-1}(\Omega^2)\ar[r]^-{\phi_*}& H_{n-1}(\Omega )\\
    &\bigoplus_{i+j=n-1}H_i(\Omega)\otimes H_j(\Omega)\ar[u]^-\times_\cong\ar[r]_-{id\times\vartheta_*}&\bigoplus_{i+j=n-1}H_i(\Omega)\otimes H_j(\Omega)\ar[u]_-\times^\cong\ar[ru]_-\star&\,
    }
\end{equation*}
shows that
\begin{equation*}
    0=a\left(x\star1+1\star \vartheta_*(x)\right)=a\left(x+\vartheta_*(x)\right)
\end{equation*}
holds since $\vartheta_*(1)=1$ and hence, $\vartheta_*(x)=-x$.
\end{proof}

Lets us see how $\vartheta_*$ comports with the Pontrjagin multiplication of the generators of $H_*(\Omega S^n)$:

\begin{lem}\label{pontrjagin generator}
Let $\vartheta$ be the orientation reversal of loops on the based loop space $\Omega S^n$, $n>2$. Then in the Pontrjagin ring $(H_*(\Omega S^n;\Z),\star)\cong\Z[x]$, where $|x|=n-1$, we have
\begin{equation*}
  \vartheta_*(x^{\star k})=\begin{cases}\big(\vartheta_*(x)\big)^{\star k}=(-1)^k x^{\star k}, & \text{if }n\text{ is odd or }(k-1)k\equiv0\text{ mod }4\\ -\big(\vartheta_*(x)\big)^{\star k}=(-1)^{k+1} x^{\star k}, & \text{if }n\text{ is even and }(k-1)k\not\equiv0\text{ mod }4\end{cases}.
\end{equation*}
\end{lem}

\begin{proof}
Similar to the case of free loops, for $\Omega=\Omega M$ we have that for the concatenation $\phi$
\begin{equation*}
    \vartheta\circ\phi=\phi\circ\left(\mathrm{T}\circ(\vartheta\times\vartheta)\right)
\end{equation*}
holds. Here $\mathrm{T}:\Omega\times\Omega\rightarrow\Omega\times\Omega, (\alpha,\beta)\mapsto (\beta,\alpha)$. Hence, we get the commutative diagram (any coefficients)
\begin{equation*}
\xymatrixrowsep{1.5pc}
 \xymatrixcolsep{4pc}\xymatrix{
   H_i(\Omega)\otimes H_j(\Omega)\ar[r]_-\times\ar[d]_-{\vartheta_*\otimes\vartheta_*}&H_{i+j}(\Omega\times\Omega)\ar[r]_-{\phi_*}\ar[d]_-{(\vartheta\times\vartheta)_*}&H_{i+j}(\Omega)\ar[dd]_-{\vartheta_*}\\
   H_i(\Omega)\otimes H_j(\Omega)\ar[r]_-\times\ar[d]_-{a\otimes b\mapsto b\otimes a}&H_{i+j}(\Omega\times\Omega)\ar[d]_-{\mathrm{T}_*}&\\
   H_j(\Omega)\otimes H_i(\Omega)\ar[r]_-{(-1)^{ij}\times}&H_{i+j}(\Omega\times\Omega)\ar[r]_-{\phi_*}&H_{i+j}(\Omega),\\
     }  
\end{equation*}
i.e. $(-1)^{ij}\vartheta_*(b)\star\vartheta_*(a)=\vartheta_*(a\star b)$ for all $a\in H_i(\Omega),b\in H_j(\Omega)$. If $M=S^n$ the Pontrjagin algebra with integer coefficients is a polynomial algebra and thus in particular commutative, hence we can also write
\begin{equation*}
 (-1)^{ij}\vartheta_*(a)\star\vartheta_*(b)=\vartheta_*(a\star b)
\end{equation*}
for $a\in H_i(\Omega S^n;\Z),b\in H_j(\Omega S^n;\Z)$. For any class $a\in H_i(\Omega S^n;\Z)$ we have $\vartheta_*(a)=\pm a$ as $\vartheta$ is an involution and $H_i(\Omega S^n;\Z)$ is isomorphic to either $\Z$ or $0$. Thus if above $a=b$ we have
\begin{equation*}
  \vartheta_*(a^{\star 2})=\begin{cases} \big(\vartheta_*(a)\big)^{\star 2}=a^{\star 2}, & \text{if }|a|\text{ is even}\\ -\big(\vartheta_*(a)\big)^{\star 2}=-a^{\star 2}, & \text{if }|a|\text{ is odd}\end{cases}.
\end{equation*}
 The assertion of the lemma follows by iteratively applying the above and setting $a=x$.
\end{proof}

We now use the above lemmata to see how $\vartheta$ acts on the generators of the Chas-Sullivan algebra of spheres.

\begin{prp}
\begin{itemize}
\item For $n$ odd the Chas-Sullivan algebra $\big(H_*(\Lambda S^n;\Z);\ast\big)$ has two generators $A$ and $U$ and a unit $E$ (see Section \ref{CS spheres}) and we have
\begin{align*}
\vartheta_*(A)=A,\quad
\vartheta_*(E)=E,\quad
\vartheta_*(U)=-U.
\end{align*}

\item for $n$ even the Chas-Sullivan algebra $\big(H_*(\Lambda S^n;\Z);\ast\big)$ has three generators $A,\sigma_1$ and $\Theta$ and a unit $E$ (see Section \ref{CS spheres}) and we have
\begin{align*}
\vartheta_*(A)=A,\quad
\vartheta_*(E)=E,\quad
\vartheta_*(\sigma_1)=-\sigma_1,\quad
\vartheta_*(\Theta)=-\Theta.
\end{align*}
\end{itemize}
\end{prp}

\begin{proof}
Since the inclusion $S^n\hookrightarrow \Lambda$ as constant curves is a $\vartheta$-equivariant section of $\Lambda\xrightarrow{ev_0} S^n$, it follows that the generators $A$ and $E$ are fixed und $\vartheta_*$.\\
For the generators $\sigma_r\in H_{(2r-1)(n-1)}(\Lambda;\Z)\cong\Z$ we have
\begin{equation*}
 \sigma_r\coloneqq\sigma_1\ast \Theta^{\ast (r-1)}=\Theta^{\ast (r-1)}\ast\sigma_1.
\end{equation*} 
Here $\Theta\in H_{3n-2}(\Lambda)$ is a designated generator (compare \cite[Lemma 5.4]{hingston2013}). Furthermore, form the homotopy sequence of the fibration $\Omega S^n\xrightarrow{j}\Lambda S^n\xrightarrow{ev_0}S^n$, for the generator $x\in H_{n-1}(\Omega S^n;\Z)$ of the Pontrjagin ring we have $j_*(x)=\sigma_1$. The sequence implies that
\begin{equation*}
 \xymatrix{
 H_{n-1}(\Omega S^n;\Z)\cong\pi_{n-1}(\Omega S^n)\ar[r]^-{j_*}&\pi_{n-1}(\Lambda S^n)\cong H_{n-1}(\Lambda S^n;\Z)
 }
\end{equation*}
is surjective. Hence $j_*$ maps generators to generators in degree $n-1$. The isomorphisms $H_{n-1}(\Omega S^n;\Z)\cong\pi_{n-1}(\Omega S^n)$ and $\pi_{n-1}(\Lambda S^n)\cong\ H_{n-1}(\Lambda S^n;\Z)$ also follow from this homotopy sequence: It implies that $0=\pi_k(\Omega S^n)=\pi_k(\Lambda S^n)$ for $k\leq n-2$, i.e. that $\Omega S^n$ and $\Lambda S^n$ are $(n-2)$-connected. This in turn implies, via the Hurewicz theorem, that the Hurewicz map is an isomorphism in degree $n-1$  if $n-1$ is at least 2.

We now use the relations between the algebras of the spaces in the fibration $\Omega\to\Lambda\to M$ of Proposition \ref{alg hom CS Pont}.
More precisely, we use the following three relations given in \cite[Section 9.3]{goresky2009} or in \cite[Section 6, Equations 27-29]{hingston2013}:
\begin{align}
 j_!(a\ast b)&=j_!(a)\star j_!(b),\label{1}\\
 j_*(y)\ast a&=j_*\big(y\star j_!(a)\big),\label{2}\\
 j_*\big(j_!(a)\big)&=A\ast a, \label{3}
\end{align}
where $A$ is a generator of $H_0(\Lambda)$ and $\star$ denotes the Pontrjagin product. The third equation follows from the second noting that in the Pontrjagin algebra the class of a constant loop is a unit. Using these equations for $M=S^n$ we obtain:
\begin{itemize}
 \item for $n$ odd we have $U^{\ast2}=\Theta$, where $U\in H_{2n-1}(\Lambda S^n;\Z)$ is a generator. Hence, since $\sigma_1=A\ast U$, using equations \eqref{3} and \eqref{1}, we have  
  \begin{align*}
  \sigma_r&=\sigma_1\ast\Theta^{\ast(r-1)}=\sigma_1\ast U^{\ast(2r-2)}=A\ast U^{\ast(2r-1)}=j_*j_!(U^{\ast(2r-1)})\\
  &=j_*j_!(U)^{\star (2r-1)}=j_*(x^{\star(2r-1)})
 \end{align*}

since $j_!(U)=x$ (compare \cite[Section 6]{hingston2013}). Note that $j_!(U)=x$ must hold, as the above for $r=1$ is $j_*(x)=\sigma_1=j_*\big(j_!(U)\big)$. Since $|x|=n-1$ is even, it follows from Lemma \ref{pontrjagin generator} that 
 \begin{align*}
  \vartheta_*(\sigma_r)&=\vartheta_*\big(j_*(x^{\star(2r-1)})\big)=j_*\big(\vartheta_*(x^{\star(2r-1)})\big)=j_*\big({\vartheta_*(x)}^{\star(2r-1)}\big)\\&=j_*({-x}^{\star(2r-1)})=-j_*(x^{\star(2r-1)})=-\sigma_r.
 \end{align*}

\item for $n$ even, using the equations \eqref{2} and \eqref{1}, we only get
\begin{align*}
 \sigma_r&=\sigma_1\ast\Theta^{\ast(r-1)}=j_*(x)\ast\Theta^{\ast(r-1)}=j_*\big(x\star j_!(\Theta^{\ast (r-1)})\big)\\
 &=j_*\big(x\star j_!(\Theta)^{\star (r-1)}\big)
\end{align*}
since $H_{2n-1}(\Lambda S^n;\Z)=0$ (compare \cite[Section 6]{hingston2013}). Hence we have
 \begin{align*}
  \vartheta_*(\sigma_r)&=\vartheta_*\Big(j_*\big(x\star j_!(\Theta)^{\star(r-1)}\big)\Big)=j_*\Big(\vartheta_*\big(x\star j_!(\Theta)^{\star(r-1)}\big)\Big)\\
  &=j_*\Big(\vartheta_*(x)\star \vartheta_*\big(j_!(\Theta)^{\star(r-1)}\big)\Big)=-j_*\Big(x\star \vartheta_*\big(j_!(\Theta)^{\star(r-1)}\big)\Big)\\
  &=-j_*\Big(x\star \big(\vartheta_*\circ j_!(\Theta)\big)^{\star(r-1)}\Big)\,,
 \end{align*}
since $j_!(\Theta)$ has even degree. $\sigma_r=j_*(x\star j_!(\Theta)^{\ast(r-1)})$ being a generator implies that $j_!(\Theta)=\pm x^2$. Thus
\begin{align*}
 \vartheta_*\big(j_!(\Theta)\big)&=\vartheta_*(\pm x^{\star2})=\pm\vartheta_*(x^{\star2})=\pm(-x^{\star2})=-1(\pm x^{\star2})=-j_!(\Theta)
\end{align*}
and so
\begin{equation*}
  \vartheta_*(\sigma_r)=(-1)^rj_*\Big(x\star \big(j_!(\Theta)\big)^{\star(r-1)}\Big)=(-1)^r\sigma_r.
 \end{equation*}
\end{itemize}
For $\Theta$ we have
\begin{equation*}
  \vartheta_*(\Theta)=(-1)^{(n-1)}\Theta  
\end{equation*}
since
 \begin{itemize}
  \item  if $n$ is odd we have $\vartheta_*(\Theta)=\vartheta_*(U^{\ast2})=\vartheta_*(U)^{\ast2}$ (see proposition \ref{inv vs mult CS}) and $-A\ast U=-\sigma_1=\vartheta_*(\sigma_1)=\vartheta_*(A\ast U)=\vartheta_*(A)\ast \vartheta_*(U)=A\ast\vartheta_*(U)$. Since the homology is torsion-free for $n$ odd and with at most one generator in each degree, it follows that $\vartheta_*(U)=-U $ and hence $\vartheta_*(\Theta)=\Theta$ if $n$ is odd.
  \item if $n$ is even $\vartheta_*(\sigma_2)=\vartheta_*(\sigma_1\ast\Theta)=\vartheta_*(\sigma_1)\ast\vartheta_*(\Theta)=-\sigma_1\ast\vartheta_*(\Theta)$ (again by equation \ref{alg hom}). But we have just shown that $\vartheta_*(\sigma_2)=\sigma_2$ hence $\vartheta_*(\Theta)=-\Theta$ for even $n$.
 \end{itemize}
\end{proof}

Together with Proposition \ref{inv vs mult CS} we get (see figure \ref{fig:hom diag})

\begin{cor}
Let $\vartheta\colon\Lambda S^n\to\Lambda S^n$ be the involution that reverses the orientation of loops. Then, if $n\geq3$, for the orbit space $\Lambda/\vartheta$ of the associated $\Z_2$-action we have:
 \begin{itemize}
  \item for $n$ odd 
  \begin{equation*}
   H_i(\Lambda S^n/\vartheta;\Q)\cong\begin{cases} \Q& i=0,n \\  \Q& i=n-1+\lambda_r=2r(n-1) \text{ for }r\in\N \\  \Q& i=2n-1+\lambda_r=n+2r(n-1) \text{ for }r\in\N \\ \{0\} &\text{otherwise} \end{cases}
  \end{equation*}
\item for $n$ even
    \begin{equation*}
   H_i(\Lambda S^n/\vartheta;\Q)\cong\begin{cases} \Q& i=0,n \\  \Q& i=\lambda_r=(2r-1)(n-1) \text{ for }r\in2\N \\  \Q& i=2n-1+\lambda_r=n+2r(n-1) \text{ for }r\in2\N \\ \{0\} &\text{otherwise} \end{cases}
  \end{equation*}
 \end{itemize}
where $\lambda_r=(2r-1)(n-1)$ is the index of a $r$-fold iterated prime closed geodesic of the round sphere $(S^n,g_{st})$.
\end{cor}

\begin{proof}
Using Proposition \ref{inv vs mult CS} we compute
\begin{itemize}
\item for $n$ odd we have $\vartheta_*(U^{\ast k})=\vartheta_*(U)^{\ast k}=(-1)^kU^{\ast k}$ and hence, by Proposition \ref{reversible vs. nonreversible?}, 
\begin{align*}
q_*(U^{\ast k})\neq0&\Leftrightarrow k\text{ is even},\\
q_*(A\ast U^{\ast k})\neq0&\Leftrightarrow k\text{ is even}.
\end{align*}

\item for $n$ even we have $\vartheta_*(\Theta^{\ast k})=(-1)^k\Theta^{\ast k}$ and $\vartheta_*(\sigma_k)=(-1)^k\sigma_k$ and so, by Proposition \ref{reversible vs. nonreversible?},
\begin{align*}
q_*(\Theta^{\ast k})\neq0&\Leftrightarrow k\text{ is even},\\
q_*(\sigma_k)=q_*(\sigma_1\ast \Theta^{\ast (k-1)})\neq0&\Leftrightarrow k\text{ is even}.
\end{align*}

\end{itemize}
\end{proof}

\begin{figure}
    \centering
\begin{equation*}    
\begin{array}{c|ccccc}
\uparrow d &  &  &  & & \\ 
8n-8 &  &  & & & \text{ \ \ }\Q \\ 
&  &  &  &  & \\ 
7n-6 &  &  &  & \text{ \ \ }\Q & \\ 
&  &  &  &  & \\ 
7n-7 &  &  & & & \text{ \ \ }0  \\ 
&  &  &  &  & \\ 
6n-5 &  &  &  & \text{ \ \ }0  \\ 
&  &  &  &  & \\ 
6n-6 &  &  &  & \text{ \ \ }\Q & \\ 
&  &  &  &  & \\ 
5n-4 &  &  & \text{ \ \ }\langle q_*(U^4)\rangle=\Q&  & \\ 
&  &  &  &  & \\ 
5n-5 &  &  &  & \text{ \ \ }0 & \\ 
&  &  &  &  & \\ 
4n-3 &  &  & \text{ \ \ }0&  & \\ 
&  &  &  &  & \\ 
4n-4 &  &  & \text{ \ \ }\langle q_*(A\ast U^4)\rangle=\Q&  & \\
&  &  &  &  & \\ 
3n-2 &  & \text{ \ \ }\langle q_*(U^2)\rangle=\Q&  &  & \\ 
&  &  &  &  & \\ 
3n-3 &  &  & \text{ \ \ }0&  & \\ 
&  &  &  &  &\\ 
2n-1 &  & \text{ \ \ }0 &  &  & \\ 
&  &  &  &  & \\ 
2n-2 &  & \text{ \ \ }\langle q_*(A\ast U^2)\rangle=\Q&   & &\\ 
&  &  &  &   &\\ 
n & \text{ \ \ \ \ \ }\Q\text{ \ \ \ \ \ } &  &  & &  \\ 
&  &  &  &  & \\ 
n-1 &  & \text{ \ \ }0&  &  & \\ 
&  &  &  &   &\\ 
0 & \text{ \ \ \ \ \ \ }\Q{ \ \ \ \ \ \ } &  &  & &  \\ 
&  &  &  &  & \\ \hline
&  &  &  &  & \\ 
& 0 & 1 & \text{ \ \ \ \ \ }2 & \text{ \ \ \ \ \ }3 & \text{ \ \ \ \ \ }4\rightarrow r
\end{array}
\end{equation*}
 \caption{the homology $H_i(\Lambda S^n/\vartheta;\Q)$ for $n$ odd }
 \label{fig:hom diag}
\end{figure}
We now compute the transfer product $P_\vartheta$ associated to orientation reversal. Recall that it is given by
\begin{equation*}
P_\vartheta(a,b)={q_\vartheta}_*\big(tr_\vartheta(a)\ast tr_\vartheta(b)\big)
\end{equation*}
for $a,b\in H_*(\Lambda/\vartheta)$.
\pagebreak

\begin{thm}\label{main thm}
 Let $n>2$ and let $\vartheta$ be the orientation reversal of loops on $\Lambda S^n$. Then
 \begin{itemize}
\item for $n$ odd, there exists  a generator $\mu$ of $H_{3n-2}(\Lambda S^n/\vartheta;\Q)$ which is not nilpotent in the algebra $(H_*(\Lambda S^n/\vartheta;\Q),P_\vartheta)$. More precisely, for every $k\in\N$, $\mu^k$ is a generator of $H_{2k(n-1)+n}(\Lambda S^n/\vartheta;\Q)$. Moreover, multiplication with $\mu$, i.e.
 \begin{equation*}
  P_\vartheta(\cdot,\mu)\colon H_i(\Lambda S^n/\vartheta;\Q)\to H_{i+2n-2}(\Lambda S^n/\vartheta;\Q)
 \end{equation*}
is an isomorphism for $i\geq0$.

\item for $n$ even, there exists a generator $\eta$ of $H_{5n-4}(\Lambda S^n/\vartheta;\Q)$ which is not nilpotent in the algebra $(H_*(\Lambda S^n/\vartheta;\Q),P_\vartheta)$. More precisely, for every $k\in\N$, $\eta^k$ is a generator of $H_{4k(n-1)+n}(\Lambda S^n/\vartheta;\Q)$. Moreover, multiplication with $\eta$, i.e.
 \begin{equation*}
  P_\vartheta(\cdot,\eta)\colon H_i(\Lambda S^n/\vartheta;\Q)\to H_{i+4n-4}(\Lambda S^n/\vartheta;\Q)
 \end{equation*}
is an isomorphism for $i>0$.
\end{itemize}
\end{thm}

\begin{proof}
\begin{itemize}
\item for $n$ odd we define $\mu\coloneqq{q_\vartheta}_*(U^{\ast 2})$. Then $\mu$ is a generator of $H_{3n-2}(\Lambda/\vartheta;\Q)$.

\item for $n$ even we define $\eta\coloneqq{q_\vartheta}_*(\Theta^{\ast 2})$.
Then $\eta$ is a generator of $H_{5n-4}(\Lambda/\vartheta;\Q)$.\\

\end{itemize}

The theorem now follows immediately from Theorem \ref{alg iso}, Proposition \ref{inv vs mult CS} and the Chas-Sullivan algebra of spheres (see Section \ref{CS spheres}).
\end{proof}

An analogous result holds for the product $P_\theta$ with rational coefficients as Proposition \ref{algebra iso} shows. In fact, the transfer algebras for all $G$ conjugate to some $D_m$ have a structure as shown in the theorem, this now follows from Theorem \ref{O(2) subgroup iso}.

 \clearpage
\bibliographystyle{amsalpha}
\bibliography{bib1}
\end{document}